\documentclass[10pt,reqno]{amsart}

\usepackage{amssymb} 
\usepackage{tikz} 

\newcommand*{\DashedArrow}[1][]{\mathbin{\tikz [baseline=-0.25ex,-latex, dashed,#1] \draw [#1] (0pt,
0.5ex) -- (1.3em,0.5ex);}} 
 
\usepackage{amscd} 
\usepackage{empheq,amsmath} 
\usepackage{graphicx}
\usepackage[cmtip,all]{xy} 
\usepackage{amsmath,amsthm,amsfonts,amssymb,amscd,amsbsy}
\usepackage{graphics}
\usepackage[latin1]{inputenc}
\usepackage{enumerate}
\usepackage{comment}
\usepackage{mathtools}
\usepackage{verbatim}
\usepackage[inline, shortlabels]{enumitem}
\usepackage{commath}
\usepackage{tensor}
\usepackage[mathscr]{euscript}
\usepackage{textcomp}
\usepackage{hyperref}
\usepackage[all,cmtip]{xy}
\usepackage{graphicx}
\usepackage{supertabular}

\def\bc{\begin{center}}
\def\ec{\end{center}}
\def\px{\frac{\partial}{\partial x_1}}
\def\py{\frac{\partial}{\partial x_2}}
\def\pz{\frac{\partial}{\partial x_n}}

\newcommand{\cod}{\mathrm {\,codim}\,}

\newtheorem{thm}{Theorem}[section]

\newtheorem{prop}[thm]{Proposition}
\newtheorem{lemma}[thm]{Lemma}
\newtheorem{definition}[thm]{Definition}
\newtheorem{remark}[thm]{Remark}
\newtheorem{cor}[thm]{Corollary}
\newtheorem{example}[thm]{Example}
\newtheorem{claim}[thm]{Claim}
\newtheorem*{thmA}{Theorem A}
\newtheorem*{thmB}{Theorem B}
\newtheorem*{thmB1}{Theorem B.1}
\newtheorem*{thmB2}{Theorem B.2}

\title[Foliations on projective spaces associated]{Foliations on projective spaces associated to the affine Lie Algebra }
\author[R. C. da Costa]{Raphael Constant da Costa}
\address{Rio de Janeiro State University (UERJ), R. S\~ao Francisco Xavier, 524, Maracan\~a,  20550-900, Rio de Janeiro, Brazil.}
\email{raphaelconstant@ime.uerj.br}

\numberwithin{equation}{section}
\begin{document}
\newcommand{\rot}{\mathrm {\,rot}\,}
\newcommand{\aut}{\mathrm {\,Aut}\,}
\newcommand{\dive}{\mathrm {\,div}\,}
\newcommand{\degr}{\mathrm {\,deg}\,}
\newcommand{\sing}{\mathrm {\,Sing}\,}
\newcommand{\ld}{\mathrm {\,Ld}\,}
\newcommand{\im}{\mathrm {\,Image}\,}

\begin{abstract}
In this work, we construct some irreducible components of the space of two-dimensional holomorphic foliations on $\mathbb{P}^n$ associated to some algebraic representations of the affine Lie algebra $\mathfrak{aff}(\mathbb{C})$. We give a description of the generalized Kupka components, obtaining a classification of them in terms of the degree of the foliations, in both cases $n=3$ and $n=4$. 
\end{abstract}

\maketitle

\tableofcontents

\section{Introduction}
We consider a holomorphic foliation of dimension $k$ and degree $d$ on the projective space $\mathbb{P}^n$, $n \geq 3$. The set of those foliations, which we denote by $\mathscr{F}_k(d,n)$, has a natural structure of quasi-projective variety. In fact, such foliations are defined by an integrable ($n-k$)-form $\Omega$ on $\mathbb{C}^{n+1}$, whose coefficients are homogeneous polynomials of degree $d+1$ satisfying $i_{R_{n+1}} \Omega =0$, where $R_{n+1}$ denotes the radial vector field on $\mathbb{C}^{n+1}$. The ($n-k$)-form $\Omega$ is defined up to multiplication by a non-zero scalar, giving rise to a projective space, and the integrability condition imposes polynomial relations on that space. Finally, from the condition $\cod (\sing \Omega) \geq 2$, where $\sing (\Omega)$ denotes the singular set of $\Omega$, we identify $\mathscr{F}_k(d,n)$ with a Zariski open subset of a projective variety. A very interesting question is to describe the irreducible components of $\mathscr{F}_k(d,n)$. The known results are mostly concentrated in the codimension one case ($k=n-1$). Some of the irreducible components of $\mathscr{F}_{n-1}(d,n)$ have been described: linear pull-back \cite{3}, rational \cite{2}, logarithmic \cite{7}, generic pull-back \cite{6}, associated to the affine Lie algebra \cite{1}, rigid \cite{5} and more recently branched pull-back \cite{9}. A complete description of the irreducible components of $\mathscr{F}_{n-1}(d,n)$ is known only in low degrees. In \cite{10} it has been shown that $\mathscr{F}_{n-1}(0,n)$ has only one irreducible component, while $\mathscr{F}_{n-1}(1,n)$ consists of two irreducible components. The classification of $\mathscr{F}_{n-1}(2,n)$ was achieved by Cerveau and Lins Neto in \cite{4}, where they show that $\mathscr{F}_{n-1}(2,n)$ has six irreducible components. The literature on the irreducible components of $\mathscr{F}_k(d,n)$, $1\leq k <n-1$, is not as extensive in comparison with the codimension one case. Some results in this direction can be found in \cite{5} and \cite{20}. The classification of $\mathscr{F}_k(0,n)$ was given in \cite[Theorem 3.8]{18}, while a complete description of $\mathscr{F}_k(1,n)$ was obtained in \cite[Theorem 6.2 and Corollary 6.3]{19}.

In this paper, we construct and classify certain components of $\mathscr{F}_{2}(d,n)$ associated to the affine Lie Algebra $\mathfrak{aff}(\mathbb{C})=\langle e_1, e_2 \rangle$, where $[e_1, e_2]=e_2$. These components include those described in \cite{1}. Let $p_1>p_2>\cdots>p_n \geq 1$ be relatively prime positive integers and $S$ the diagonal vector field of $\mathbb{C}^n$ defined by $$S=p_1 x_1 \px +p_2 x_2 \py + \cdots +p_n x_n \pz.$$

Let $X$ be another polynomial vector field on $\mathbb{C}^n$ such that $[S,X]=\lambda X$, for some $\lambda \in \mathbb{Z}$. Note that if $\lambda \neq 0$, $S$ and $X$ give a representation of $\mathfrak{aff}(\mathbb{C})$ in the algebra of polynomial vector fields of $\mathbb{C}^n$. In addition, if $S$ and $X$ are linearly independent at generic points, they give rise to a dimension two algebraic foliation $\mathcal{F}=\overline{\mathcal{F}}(S,X)$ on $\mathbb{C}^n$, which is defined by the following integrable ($n-2$)-form $$\omega = i_{S} i_{X}(dx_1 \wedge dx_2 \wedge \cdots \wedge dx_n).$$

Define $\mathfrak{P}=(p_1,p_2,\ldots,p_n)$ and $$\mathcal{F}(\mathfrak{P},\lambda,d+1)=\{\mathcal{F} \in \mathscr{F}_2(d+1,n) \mid \mathcal{F}= \overline{\mathcal{F}}(S,X) \text{ in some affine chart} \}.$$

\begin{remark}
In the last definition, we can choose $X$ in such a way that the one dimensional foliation on $\mathbb{P}^n$ generated by $X$ has degree $d$, which simplifies some calculations (see Lemma \ref{lemmadeg}). This is the reason we do not adopt $\mathcal{F}(\mathfrak{P},\lambda,d)$.
\end{remark}

It turns out that $\overline{\mathcal{F}(\mathfrak{P},\lambda,d+1)}$, the Zariski closure of $\mathcal{F}(\mathfrak{P},\lambda,d+1)$, is an irreducible subvariety of $\mathscr{F}_2(d+1,n)$ (see Proposition \ref{irreducibility}). In the cases $n=3$ and $n=4$, we use $\mathcal{F}(p,q,r;\lambda,d+1)$ and  $\mathcal{F}(p,q,r,s;\lambda,d+1)$, respectively. Next we present some conditions that entail the existence of irreducible components $\overline{\mathcal{F}(\mathfrak{P},\lambda,d+1)}$ of $\mathscr{F}_2(d+1,n)$.

Let $\omega$ be a germ of integrable $(n-2)$-form defined at $p \in \mathbb{C}^n$, with $p \in \sing(\mathcal{\omega})$ and $n \geq 3$.

\begin{definition}
We say that $p$ is a \textit{weakly generalized Kupka }(WGK) singularity of $\omega$ if $\cod(\sing(d\omega)) \geq 3$, where by convention $\cod(\emptyset) = n+1$. In addition, if $\cod(\sing(d\omega)) \geq n$ we say that $p$ is a \textit{generalized Kupka }(GK) singularity. 
\end{definition}

\begin{definition}
A dimension two holomorphic foliation $\mathcal{F}$ on $\mathbb{P}^n$ is WGK (resp. GK) if all the singularities of $\mathcal{F}$ are WGK (resp. GK).
\end{definition}

The following result was proved in \cite{1}.

\begin{thm}\label{thefour}
Suppose that $\mathcal{F}(p,q,r;\lambda,d+1)$ contains some GK foliation, where $\lambda \neq 0$ and $p>q>r$ are relatively prime positive integers. Then $\overline{\mathcal{F}(p,q,r;\lambda,d+1)}$ is an irreducible component of $\mathscr{F}_2(d+1,3)$.
\end{thm}

The irreducible components $\overline{\mathcal{F}(\mathfrak{P},\lambda,d+1)} \subset \mathscr{F}_2(d+1,n)$ containing GK foliations will be called GK components.

\begin{cor}\label{corklein}
For $d \geq 1$, $\overline{\mathcal{F}(d^2+d+1,d+1,1;-1,d+1)}$ is an irreducible component of $\mathscr{F}_2(d+1,3)$ of dimension $N$, where $N=13$ if $d=1$ and $N=14$ if $d>1$. Moreover, this component is the closure of a $\mathrm{PGL}(4,\mathbb{C})$ orbit on $\mathscr{F}_2(d+1,3)$.
\end{cor}

Note that these families extend the so-called exceptional component ($d=1$), that appears originally in \cite{4}, and they consist of the only general families of GK components provided by Theorem \ref{thefour} that are known so far. 

Even when $p_1\geq p_2 \geq \cdots \geq p_n$, the construction of $\mathcal{F}(\mathfrak{P},\lambda,d+1)$ makes sense. Recently, the latter case was treated in \cite{11}. Our first result extends Theorem \ref{thefour} to higher dimensional projective spaces. Thinking $S$ as defined in an affine chart $(E \cong \mathbb{C}^n,(x_1,\ldots,x_n))$, denote by $q_0 \in \mathbb{P}^n$ the point corresponding to $0 \in E$.

\begin{thmA}\label{a}
Let $n,d,\lambda$ be integers, with $n \geq 3$ and $d \geq 1$. If $\lambda>0$ and $\mathcal{F}(\mathfrak{P},\lambda,d+1)$ contains some WGK foliation $\mathcal{F}$, where $q_0$ is a GK singularity of $\mathcal{F}$, then $\overline{\mathcal{F}(\mathfrak{P},\lambda,d+1)}$ is an irreducible component of $\mathscr{F}_2(d+1,n)$. In particular, if $\mathcal{F}(\mathfrak{P},\lambda,d+1)$ contains some GK foliation, where $\lambda \neq 0$, then $\overline{\mathcal{F}(\mathfrak{P},\lambda,d+1)}$ is an irreducible component of $\mathscr{F}_2(d+1,n)$.
\end{thmA}

\begin{remark}\label{othertype}
It is worth pointing out that there are irreducible components of $\mathscr{F}_2(d+1,n)$ given by Theorem \hyperref[a]{A} that are not GK, as $\overline{\mathcal{F}(10,8,6,1;2,3)} \subset \mathscr{F}_2(3,4)$.
\end{remark}

The proof of Theorem \hyperref[a]{A} has much in common with the proof of Theorem \ref{thefour}. The main difference is related to recent results on quasi-homogeneous singularities, previously restricted to the case of dimension 3.

Next we give a description of the components $\overline{\mathcal{F}(\mathfrak{P},\lambda,d+1)} \subset \mathscr{F}_2(d+1,n)$ provided by the second part of Theorem \hyperref[a]{A}. Loosely speaking, $\mathcal{F}(\mathfrak{P},\lambda,d+1)$ contains a GK foliation if and only if $q_0$ is a GK singularity of some $\mathcal{F} \in \mathcal{F}(\mathfrak{P},\lambda,d+1)$ and $p_1,\ldots,p_n,\lambda,d$ satisfy certain arithmetic relations.

Throughout the text, several parameters will appear, including in the next theorem. We seize the opportunity to define most of them now. Given $p_1,\ldots,p_n,\lambda,d$, by convention set $p_{n+1}=0$. Define $\lambda_1,\ldots,\lambda_n$, $\tau,\tau_1,\ldots,\tau_n$, $\overline{p}_1,\ldots,\overline{p}_n$ as follows

\begin{equation}\label{mainpar}
\begin{cases}
\lambda_1=p_1(d-1)-\lambda, \lambda_i=\lambda-p_i(d-1),i=2,\ldots,n,\\
\tau=\lambda+\sum_{k=1}^{n} p_k, \\
\tau_1=p_1(n+d)-\tau,\tau_i=\tau-p_i(n+d), i=2,\ldots,n, \\
\overline{p}_j=p_1-p_{n-j+2},j = 1,\dots,n.
\end{cases}
\end{equation} 

For $i=1,\ldots,n-1,j=1,\ldots,n$ denote by $c_{ij}$ the following condition
\begin{equation*}
c_{ij}:
\begin{cases}
p_j +\lambda= p_{i+1}d, \text{ if } j \leq i\\
p_{j+1}+\lambda=p_{i+1}d, \text{ if } j >i.
\end{cases}
\end{equation*}

Before stating the next result, it is worth mentioning that $\mathcal{F}(\mathfrak{P},\lambda,d+1)=\mathcal{F}(\overline{\mathfrak{P}},\lambda_1,d+1)$, where $\overline{\mathfrak{P}}=(\overline{p}_1,\ldots,\overline{p}_n)$ (see Proposition \ref{symmetry}). 

\begin{thmB}\label{b}
Let $l_1>\cdots>l_n$ be relatively prime positive integers, $\mu \in \mathbb{Z}$ and $d \geq 1$. Then $\overline{\mathcal{F}(\mathfrak{L},\mu,d+1)}$ is a GK component of $\mathscr{F}_2(d+1,n)$ if and only if it can be written in the form $\overline{\mathcal{F}(\mathfrak{L},\mu,d+1)}=\overline{\mathcal{F}(\mathfrak{P},\lambda,d+1)}$, such that $p_1>\cdots>p_n$ are relatively prime positive integers, $\lambda \in \mathbb{Z}_{>0}$, satisfying
\begin{enumerate}[a)]
\item 
$q_0$ is a GK singularity of some $\mathcal{F} \in \mathcal{F}(\mathfrak{P},\lambda,d+1)$
\end{enumerate}
and $p_1,\ldots,p_n,\lambda,d$ satisfy either 
\begin{enumerate}
\item [b.1)]
\begin{itemize}
\item
$c_{11},c_{22},\ldots,c_{ii},c_{i+1,i+2},c_{i+2,i+3},\ldots,c_{n-1,n}$, for some $0 \leq i \leq \lfloor \frac{n-1}{2} \rfloor$
\item 
$\tau_j \neq 0, j=2,3,\ldots,n$
\end{itemize}
\end{enumerate}
or
\begin{enumerate}
\item [b.2)]
\begin{itemize}
\item
$c_{11},c_{22},\ldots,c_{i-2,i-2},\lambda=p_i (d-1),c_{i,i+1},c_{i+1,i+2},\ldots,c_{n-1,n}$, for some $2 \leq i \leq \lfloor \frac{n+2}{2} \rfloor$
\item 
$\tau_j \neq 0, j \in \{2,3,\ldots,n\} \setminus \{i\}$
\end{itemize}
\end{enumerate}

Moreover, if $\mathfrak{L}=\mathfrak{P}$ and $\mu=\lambda$ do not hold, certainly $\mathfrak{L}=\overline{\mathfrak{P}}$ and $\mu=\lambda_1$.
\end{thmB}

For each $d \geq 1$, we have at least one irreducible component of $\mathscr{F}_2(d+1,n)$ described by Theorem \hyperref[b]{B}. As we will see in Corollary \ref{gen-klein-lie}, for $$\mathfrak{L}=(d^{n-1}+\cdots+1,d^{n-2}+\cdots+1,\ldots,d+1,1),$$ $\overline{\mathcal{F}(\mathfrak{L},-1,d+1)}$ is an irreducible component of $\mathscr{F}_2(d+1,n)$, extending the irreducible components of Corollary \ref{corklein}. Moreover, when $d=1$ it is the only GK irreducible component $\overline{\mathcal{F}(\mathfrak{P},\lambda,2)} \subset \mathscr{F}_2(2,n)$. This is the reason we sometimes focus on the case $d \geq 2$. We point out that for $d=1$ this irreducible component was established in \cite{5}. For $d \geq 2$ it is new.

For the cases $n=3$ and $n=4$, we can exhibit the GK components in a more explicit way, as follows.

\begin{thmB1}\label{thmB1}
Let $p>q>r$ be relative prime positive integers. \\$\overline{\mathcal{F}(p,q,r;\lambda,d+1)}$ is a GK component of $\mathscr{F}_2(d+1,3)$, $d \geq 2$, if and only if either $p,q,r,\lambda,d$ or $\overline{p},\overline{q},\overline{r},\lambda_1,d$ satisfy one of the following relations
\begin{enumerate}[(a)]
\item
$p>q=m(d+1)>r=md,\lambda=md^2,\gcd(p,m)=1$, $p$ divides either $d^2$ or $d^2+d+1$;
\item
$p=d>q=r+1>r,\lambda=dr$;
\item
$p=kd>q=md+k>r=md,\lambda=md^2,\gcd(k,m)=1,k \text{ divides } d+1$;
\item
$p>q=md>r=m(d-1),\lambda=m(d^2-d),\gcd(p,m)=1$, $p$ divides either $d^2-d$, or $d^2$, or $d^2-1$.
\end{enumerate}
\end{thmB1}

We make some comments about Theorem \hyperref[thmB1]{B.1}. It provides a classification of the irreducible components given by Theorem \ref{thefour} in terms of the degree of foliations. In fact, for each $d \geq 2$, we can find (in algorithmic fashion) all the GK components $\overline{\mathcal{F}(p,q,r;\lambda,d+1)}$ of $\mathscr{F}_2(d+1,3)$. For example, we do so in Corollary \ref{corsomecomponents} for the cases $d=2$ and $d=3$, obtaining irreducible components of $\mathscr{F}_2(3,3)$ and $\mathscr{F}_2(4,3)$ which had been unknown until then. Corollary \ref{answerproblem} gives a negative answer to Problem 1 of \cite{1}, which asks whether, given three positive integers $p>q>r\geq1$, we can find $(\lambda,d)$ such that $\mathcal{F}(p,q,r;\lambda,d+1)$ is a GK family. Finally, we describe in Corollary \ref{corfamilies} new families of irreducible components like those of Corollary \ref{corklein}.

For the case $n=4$, we have an equivalent result. 

\begin{thmB2}\label{thmB2}
Let $p>q>r>s$ be relatively prime positive integers. \\ $\overline{\mathcal{F}(p,q,r,s;\lambda,d+1)}$ is a GK component of $\mathscr{F}_2(d+1,4)$, $d \geq 2$, if and only if either $p,q,r,s,\lambda,d$ or $\overline{p},\overline{q},\overline{r},\overline{s},\lambda_1,d$ satisfy one of the following relations
\begin{enumerate}[(a)]
\item
$p>q=m(d^2+d+1)>r=m(d^2+d)>s=md^2,\lambda=md^3,\gcd(p,m)=1$, $p$ divides either $d^3$ or $d^3+d^2+d+1$;
\item
$p=kd>q=md+k>r=m(d+1)>s=md,\lambda=md^2,\gcd(k,m)=1$, either $k$ divides $d$, or $kd$ divides $m(d^2+d)+k$ (which implies $k=jd$ where $j$ divides $d+1$), or $d$ divides $m$ and $k$ divides $d^2+d+1$, or $k$ divides $d+1$ and $\gcd(\frac{m(d+1)}{k},d)=1$;
\item
$p>q=md^2>r=m(d^2-1)>s=m(d^2-d),\lambda=m(d^3-d^2),\gcd(p,m)=1$, $p$ divides either $d^3-d^2$, or $d^3$, or $d^3 -1$;
\item
$p=kd>q=m(d-1)+k>r=md>s=m(d-1),\lambda=m(d^2-d),\gcd(k,m)=1$, either $k$ divides $d-1$, or $k$ divides $d$, or $d$ divides $m$ and $k$ divides $d^2-1$.
\end{enumerate}
\end{thmB2}

The same comments on Theorem \hyperref[thmB1]{B.1} apply to Theorem \hyperref[thmB2]{B.2}. We exhibit in Corollary \ref{cod2comp}, for example, new irreducible components of $\mathscr{F}_2(3,4)$. 

The paper is organized as follows. In Section 2 we list some basics properties of the foliations in $\mathcal{F}(\mathfrak{P},\lambda,d+1)$, which will be used throughout. For the sake of completeness, we determine the tangent sheaf of their foliations and the dimension of these subvarieties as well. In section 3 we recall basic facts concerning the stability of quasi-homogeneous singularities, settling a key result to obtain Theorem \hyperref[b]{B}. Theorem \hyperref[a]{A} is also proved in this section. Finally, Section 4 is dedicated to the proofs of Theorems \hyperref[b]{B}, \hyperref[thmB1]{B.1}, \hyperref[thmB2]{B.2} and some consequences.

\section{Preliminaries}
Throughout this paper, given a polynomial vector field $Z$ on $\mathbb{C}^n$, we denote by $Z=\hat{Z}_0+\hat{Z}_1+\cdots+\hat{Z}_k,\deg(\hat{Z}_i)=i$, $i=0,\ldots,k$, its decomposition into homogeneous polynomial vector fields. In parallel, we write $\omega=\hat{\omega}_0+\cdots+\hat{\omega}_k$ for a polynomial ($n-2$)-form $\omega$ on $\mathbb{C}^n$.

\subsection{Quasi-homogeneous vector fields}
Consider the diagonal vector field $$S=p_1 x_1 \px +p_2 x_2 \py + \cdots +p_n x_n \pz,$$ where $p_1 \geq p_2 \geq \cdots \geq p_n$ are integers (not necessarily positive). 

The next result is an adapted version of Proposition 4.2.1 of \cite{8}. The proof of the original proposition still holds.

\begin{prop}\label{prop-qh-vf}
Let $X \neq 0$ be a holomorphic vector field on $\mathbb{C}^n$, where $[S,X]=\lambda.X$. Then
\begin{enumerate}[(a)]
\item
$\lambda \in \mathbb{Z}$.
\item
$\ld(S,X)=\{z \in \mathbb{C}^n \mid S(z) \text{ and } X(z) \text{ are linearly dependent} \}$ is a union of orbits of the action induced by the vector field $S$.
\end{enumerate}

Additionally, if $p_n \geq 1$ then
\begin{enumerate}[(a)]
\setcounter{enumi}{2}
\item
$\lambda \geq -p_1$ and $X$ is a polynomial vector field.
\item
If $0 \in \mathbb{C}^n$ is an isolated singularity of $X$, then the Milnor number of $X$ at $0$ is given by
\begin{center}
$m(X,0)=\frac{\prod_{j=1}^{n}(p_j+\lambda)}{\prod_{j=1}^{n}p_j}$.
\end{center}
\end{enumerate}
\end{prop}

By Proposition \ref{prop-qh-vf} (a), there is no loss assuming that $p_1,\ldots,p_n$ are relatively prime in the definition of $\mathcal{F}(\mathfrak{P},\lambda,d+1)$. The relation $[S,X]=\lambda.X$ can be given in some equivalent ways, as follows.

\begin{prop}\label{closerlook}
Let $X=\sum_{j=1}^{n} X_j(z){\partial}/{\partial z_j}$ be a holomorphic vector field on $\mathbb{C}^n$. Then the following are equivalent
\begin{enumerate}[(a)]
\item
$\left[S,X\right]=\lambda.X$.
\item
$X_j \left( t^{p_1} . z_1, \dots, t^{p_n} . z_n \right)=t^{p_j+\lambda} . X_j(z_1,\dots,z_n), \forall 1 \leq j \leq n, \forall t \in \mathbb{C}$.
\item
Write $X_j=\sum_{j \sigma} a_{j \sigma} z^{\sigma},1\leq j\leq n$, where $a_{j \sigma} \in \mathbb{C}$ and for $\sigma=(\sigma_1,\ldots,\sigma_n)$, $z^{\sigma}=z_{1}^{\sigma_1}\cdots z_{n}^{\sigma_n}$. If 
$a_{j \sigma} \neq 0$, then $\sum_{k=1}^n p_k . \sigma_k = p_j + \lambda$.
\end{enumerate}
\end{prop}

For example, if $p_j=1, 1 \leq j \leq n$, then $S=R_n$ is the radial vector field on $\mathbb{C}^n$ and the equality $[S,X]=\lambda.X$ implies that $X$ is a homogeneous polynomial vector field of degree $\lambda+1$.

\begin{remark}\label{remark 1}
Let $X$ be a holomorphic vector field on $\mathbb{C}^n$, satisfying $[S,X]=\lambda.X$. Assume that $p_n \geq 1$. If $0 \in \mathbb{C}^n$ is an isolated singularity of $X$, then $\lambda \geq 0$. If $X(0) \neq 0$, then $\lambda<0$. This is an immediate consequence of Propositions \ref{prop-qh-vf} and \ref{closerlook} above.
\end{remark}

\subsection{Some facts about foliations in $\mathcal{F}(\mathfrak{P},\lambda,d+1)$}
Hereafter we assume that $$p_1 > p_2 > \dots > p_n \geq 1.$$

Let $\mathcal{F}$ be some foliation of $\mathcal{F}(\mathfrak{P},\lambda,d+1)$. By definition, $\mathcal{F}$ is given by the following ($n-2$)-form $$\omega=i_S i_X (\nu_n),[S,X]=\lambda.X,\nu_n=dx_1 \wedge dx_2 \wedge \cdots \wedge dx_n,$$ in some affine coordinate system $(E_0,(x_1,\ldots,x_n))$, that for now we assume $$E_0= \{(x_1:\cdots:x_n:1) \in \mathbb{P}^n|(x_1,\cdots,x_n) \in \mathbb{C}^n\}.$$

As $d\omega$ is a ($n-1$)-form, there exists a vector field $Y$ such that $d \omega=i_Y(\nu_n)$. The latter is called the rotational of $\omega$, and denoted by $Y=\rot(\omega)$. Using Cartan\textquotesingle s formulas, we get $$Y=\rot(\omega)=\tau. X-\dive(X). S,$$ where $\tau=\lambda+\sum_{i=1}^n p_i$ and writing $X=\sum_{i=1}^n X_i \partial / \partial x_i$, $\dive(X)=\sum_{i=1}^n \frac{\partial X_i}{\partial x_i}$.

By Proposition \ref{prop-qh-vf} (c), we see that $\tau>0$. Using the above expression for $Y$, one verifies that 
\begin{equation}\label{eqrot}
[S,Y]=\lambda.Y,\omega = \frac{1}{\tau} i_{S}i_{Y} (\nu_n),\dive(Y) =0. 
\end{equation}

\begin{lemma}\label{lemmadeg}
Let $\mathcal{F}$ be a foliation in $\mathcal{F}(\mathfrak{P},\lambda,d+1)$, given by $\omega=i_Si_X(\nu_n)$ on $E_0$. Then $X$ can be chosen in such a way that $\deg(\mathcal{G}_X)=d$, where $\mathcal{G}_X$ denotes the one dimensional foliation on $\mathbb{P}^n$ defined by $X$ on $E_0$.
\end{lemma}
\begin{proof}
As $\deg(\mathcal{F})=d+1$, we can write on $E_0$ 
\begin{align*}
\omega=i_{S}i_{X}(\nu_n)= \hat{\omega}_0+\hat{\omega}_1+\cdots+\hat{\omega}_{d+2}, i_{R_n}(\hat{\omega}_{d+2})=0.
\end{align*}

Therefore $Y=\rot(\omega)=\hat{Y}_0+\hat{Y}_1+\cdots+\hat{Y}_{d+1}
$. The relation $i_{R_n}(\hat{\omega}_{d+2})=0$ implies that $i_{R_n}i_{S}i_{\hat{Y}_{d+1}} (\nu_n)=0$. Since $\sing(i_Si_{R_n}(\nu_n))$ is a union of lines, in particular has codimension greater than two, it follows from the last equality and Hartog\textquotesingle s Theorem that there exist holomorphic functions $f$ and $g$ on $E_0$ such that $\hat{Y}_{d+1}=f.S+h.R_n$. As $\hat{Y}_{d+1}, R, S$ are homogeneous, we can assume that $f=f_d$ and $g=g_d$ are homogeneous polynomials of degree $d$.

Finally, define $\overline{X}=\frac{Y-f_{d} . S}{\tau}$. Note that $\omega=i_S i_{\overline{X}}(\nu_n)$ and $\deg(\mathcal{G}_{\overline{X}})=d$.
\end{proof}

Denote by $\mathcal{P}_n$ the space of polynomial vector fields on $\mathbb{C}^n$. Consider the following finite-dimensional vector space over $\mathbb{C}$
\begin{align*}
W_0 = \{ Y \in \mathcal{P}_n \mid [S,Y]=\lambda Y, \dive(Y) \equiv 0,  \deg(Y) \leq d+1, i_{R_n}i_{S}i_{\hat{Y}_{d+1}} (\nu_n) \equiv 0 \}.
\end{align*}

From (\ref{eqrot}) and the proof of Lemma \ref{lemmadeg}, if $\omega=i_S i_X (\nu_n)$ defines some foliation $\mathcal{F} \in \mathcal{F}(\mathfrak{P},\lambda,d+1)$ on $E_0$, then $Y=\rot(\omega) \in W_0$. Reciprocally, given $Y \in W_0$, setting $\omega_Y=\frac{1}{\tau}i_Si_Y (\nu_n)$ we have that $Y=\rot(\omega_Y)$. 

In other words, $W_0$ is nothing more than the ambient space of $Y=\rot(\omega_Y)$, whenever $\omega_Y$ defines a foliation $\mathcal{F} \in \mathcal{F}(\mathfrak{P},\lambda,d+1)$ on $E_0$. For example, if Theorem \hyperref[b]{B} (a) holds true, there exists some $Y \in W_0$ such that $0$ is a GK singularity of $Y$.

Denote by $V_0=\mathbb{P}(W_0)=\{[Y] \mid Y \in W_0, Y \neq 0\}$ the projectivization of $W_0$ and $\aut(\mathbb{P}^n)$ the group of automorphisms of $\mathbb{P}^n$. By definition of $\mathcal{F}(\mathfrak{P},\lambda,d+1)$, there is a rational map 
\begin{equation}\label{Phi}
\Phi: V_0 \times \aut(\mathbb{P}^n)  \DashedArrow[->,densely dashed    ] \mathscr{F}_2(d+1,n)
\end{equation}
given by $\Phi([Y],T)=T^{*}\mathcal{F}(S,Y)$, where $\mathcal{F}(S,Y) \in \mathscr{F}_2(d+1,n)$ is the foliation defined by $\omega_Y$ on $E_0$ and $\im(\Phi)=\mathcal{F}(\mathfrak{P},\lambda,d+1)$. As the domain of $\Phi$ is irreducible, we have the following result

\begin{prop}\label{irreducibility}
$\overline{\mathcal{F}(\mathfrak{P},\lambda,d+1)}$ is an irreducible subvariety of $\mathscr{F}_2(d+1,n)$.
\end{prop}

\begin{prop}\label{splits}
Given $\mathcal{F} \in \mathcal{F}(\mathfrak{P},\lambda,d+1)$, then the tangent sheaf of $\mathcal{F}$ splits as $\mathcal{T}\mathcal{F}=\mathcal{O} \oplus \mathcal{O}(1-d)$.
\end{prop}
\begin{proof}
Let $\Omega$ be a homogeneous ($n-2$)-form of degree $d+2$ defining $\mathcal{F}$ in homogeneous coordinates, whose restriction to $E_0$ is $\omega=i_S i_X (\nu_n)$, where $\deg(\mathcal{G}_X)=d$.

Let $Z$ be a homogeneous vector field of degree $d$ on $\mathbb{C}^{n+1}$ defining $\mathcal{G}_X$ in homogeneous coordinates. Set $\Omega_1=i_{R_{n+1}} i_S i_Z (\nu_{n+1})$, where in the definition of $\Omega_1$ we consider $S$ as a vector field on $\mathbb{C}^{n+1}$. We have that $\Omega_1|_{E_0}=\mu. \omega$, for some $\mu \in \mathbb{C}^{*}$ (see  Proposition 4.1.2 of \cite{8}). 
Since $\deg(\Omega_1)=d+2$, $\Omega_1$ also defines $\mathcal{F}$ in homogeneous coordinates. This concludes the proof (see \S2.2 of \cite{5}).
\end{proof}

Next we will obtain expressions for  $\mathcal{F} \in \mathcal{F}(\mathfrak{P},\lambda,d+1)$ in other affine coordinate systems. For example, in  $$E_1= \{(1:u_n:u_{n-1}:\cdots:u_1)|(u_1,\cdots,u_n) \in \mathbb{C}^n\},$$ $S$ is given by $-S_1=- \sum_{j=1}^{n} \overline{p}_j u_j\partial / \partial u_j$, where $\overline{p}_j=p_1-p_{n-j+2},j = 1,\dots,n$.

Observe that $\overline{p}_1=p_1>\overline{p}_2> \cdots > \overline{p}_n$. If $\deg(\mathcal{G}_{X})=d$, $X$ has a pole of order $d-1$ at $u_1=0$ and we can write $X=\frac{X_1}{u_1^{d-1}}$, where $X_1$ defines $\mathcal{G}_{X}$ on $E_1$.
The vector field $S_1=-S$ on $E_1$ has positive eigenvalues and it will be considered on this chart. Hence
\begin{align*}
[S_1,X_1]=[-S,u_{1}^{d-1}.X]=\lambda_1 . X_1,
\end{align*}
where $\lambda_1=p_1(d-1)-\lambda$ and $\omega_1=i_{S_1}i_{X_1}(du_1 \wedge du_2 \wedge \cdots \wedge du_n)$ defines $\mathcal{F}$ on $E_1$. If $Y_1=\rot(\omega_1)$, we can write in $E_1$ similar expressions as (\ref{eqrot}).

We can proceed equally in other charts, as summarized in the following proposition (Recall the parameters (\ref{mainpar})). 

\begin{prop}\label{charts}
Given $\mathcal{F} \in \mathcal{F}(\mathfrak{P},\lambda,d+1)$, there exist affine coordinate systems $(E_i,(x_1,\ldots,x_n)), i=0,\ldots,n$, such that $\mathbb{P}^n=E_0 \cup \cdots \cup E_n$ and
\begin{enumerate}[(a)]
\item
On $E_i,i=0,\ldots,n$, $\mathcal{F}$ is defined by $\omega_i=i_{S_i} i_{X_i} (\nu_n)$, $[S_i,X_i]=\lambda_i. X_i$ ($\lambda_0=\lambda$). If $Y_i=\rot(\omega_i)$, then
\begin{align*}
Y_i=\tau_i.X_i -\dive(X_i). S_i,[S_i,Y_i]=\lambda_i. Y_i, \tau_i. \omega_i=i_{S_i} i_{Y_i} (\nu_n).
\end{align*}
Since $\tau_0=\tau,\tau_1 \neq 0$, it follows that $\dive(Y_i) \equiv 0,i=0,1$.
\item 
$S_0=S$, $S_1=\overline{p}_1 x_1 \partial / \partial x_1 +\cdots+\overline{p}_n x_n \partial / \partial x_n$ and for $i=2,\ldots,n$, writing $S_i=\sum_{j=1}^{n}\rho_j x_j \partial / \partial x_j$, we have
\begin{center}
$\rho_1=p_1-p_i>\cdots>\rho_{i-1}=p_{i-1}-p_i>0>\rho_i=p_{i+1}-p_i>\cdots>\rho_{n-1}=p_n-p_i>\rho_n=-p_i.$
\end{center}
\item
Up to a linear automorphism of $\mathbb{P}^n$, we can assume that 
\begin{align*}
E_0&= \{(x_1:\cdots:x_n:1)|(x_1,\cdots,x_n) \in \mathbb{C}^n\}, \\
E_1&= \{(1:x_n:x_{n-1}:\cdots:x_1)|(x_1,\cdots,x_n) \in \mathbb{C}^n\},  \\ 
E_i&= \{(x_1:\cdots:x_{i-1}:1:x_i:\cdots:x_n)|(x_1,\cdots,x_n) \in \mathbb{C}^n\}, i \in \{2,3,\dots,n\}.
\end{align*}
\end{enumerate}
\end{prop}
\begin{remark}
We sort the singularities $q_0,\ldots,q_n$ of $S$, thought as a global vector field on $\mathbb{P}^n$, as the points corresponding to $0 \in E_i, i=0,\ldots,n$, respectively.
\end{remark}

From now on, we think the foliations in $\mathcal{F}(\mathfrak{P},\lambda,d+1)$ endowed with the parameters of Proposition \ref{charts}.

\begin{prop}\label{symmetry}
Assume that $p_1> \cdots>p_n$, $l_1>\cdots>l_n$ are relatively prime positive integers, $\lambda,\mu \in \mathbb{Z}$. Then $\overline{\mathcal{F}(\mathfrak{P},\lambda,d+1)}=\overline{\mathcal{F}(\mathfrak{L},\mu,d+1)}$ if and only if either $\mathfrak{L}=\mathfrak{P},\mu=\lambda$ or $\mathfrak{L}=\overline{\mathfrak{P}},\mu=\lambda_1$, where $\overline{\mathfrak{P}}=(\overline{p}_1,\ldots,\overline{p}_n)$.
\end{prop}
\begin{proof}
It follows from Proposition \ref{charts} that $\mathcal{F}(\mathfrak{P},\lambda,d+1)=\mathcal{F}(\overline{\mathfrak{P}},\lambda_1,d+1)$, which immediately ensures the backward direction of the proposition. For the other direction, let $\alpha_0:\mathbb{C}^n \to \mathbb{P}^n$ be the affine coordinate chart given by $\alpha_0(x_1,\ldots,x_n)=(x_1:\cdots:x_n:1)$. Recall that there exists a one-to-one correspondence between $\aut(\mathbb{P}^n)$ and the set of affine coordinate charts $\mathcal{C}$, which associates $T \in \aut(\mathbb{P}^n)$ to $T \circ \alpha_0 \in \mathcal{C}$. For $\alpha \in \mathcal{C}$, denote by $T_\alpha$ the element of $\aut(\mathbb{P}^n)$ inducing $\alpha$.

Given $\mathcal{F} \in \mathcal{F}(\mathfrak{P},\lambda,d+1)$ and $\Omega$ a homogeneous form defining $\mathcal{F}$, by definition and Lemma \ref{lemmadeg} there is some $\beta \in \mathcal{C}$ such that $\beta^{*}\Omega=\omega=i_S i_X (\nu_n),[S,X]=\lambda.X$ and $\deg(\mathcal{G}_X)=d$. We use the following results

\begin{enumerate}[leftmargin=*]
\item
Denote by $\mathcal{V}$ the set of holomorphic vector fields on $\mathbb{P}^n$, $\mathcal{D}_{+}$ the set of diagonal vector fields $W=k_1 x_1\partial/\partial x_1+\cdots+k_nx_n\partial/\partial x_n$ on $\mathbb{C}^n$, where $k_1>\ldots>k_n$ are relatively prime positive integers and $\mathcal{D}=\mathcal{D}_{+}\cup\{-W \mid W \in \mathcal{D}_{+}\}$. Suppose that $Z \in \mathcal{V}$ is such that $\alpha_0^{*}Z=S$. If $\alpha \in \mathcal{C}$ satisfies $\alpha^{*}Z=W \in \mathcal{D}$, then either $W=S$ and $T_\alpha$ is given by a diagonal element of $\aut(\mathbb{P}^n)=\mathrm{PGL}(n+1,\mathbb{C})$, or $W=-S_1$ and $T_\alpha$ is given by a secondary diagonal element of $\aut(\mathbb{P}^n)$, where $S_1$ is as in Proposition \ref{charts} (b).
\item
If $\omega =i_S i_{\tilde{X}}(\nu_n)$, with $[S,\tilde{X}]=\tilde{\lambda}.\tilde{X}$, then $\tilde{\lambda}=\lambda$.
\item For $\alpha \in \mathcal{C}$ and $W \in \mathcal{D}$, denote by $W_\alpha$ the element of $\mathcal{V}$ such that $\alpha^{*} W_\alpha=W$. Given $W \in \mathcal{D}_{+}$ and $\xi \in \mathbb{Z}$, suppose that $\alpha \in \mathcal{C}$ is such that $\alpha^{*} \Omega=i_Wi_{\tilde{X}}(\nu_n),[W,\tilde{X}]=\xi. \tilde{X}$ and $W_{\alpha}=\pm S_{\beta}$. If $W_{\alpha}=S_{\beta}$ then $W=S$ and $\xi=\lambda$, while if $W_{\alpha}=-S_{\beta}$ then $W=S_1$ and $\xi=\lambda_1$. In fact, it suffices to check for $\beta=\alpha_0$, and it follows from (1) and (2).
\end{enumerate}

As $\overline{\mathcal{F}(\mathfrak{P},\lambda,d+1)}=\overline{\mathcal{F}(\mathfrak{L},\mu,d+1)}$, we can assume that $\mathcal{F} \in \mathcal{F}(\mathfrak{P},\lambda,d+1)$ is generic and $\mathcal{F} \in \mathcal{F}(\mathfrak{L},\mu,d+1)$. By definition there is $\gamma \in \mathcal{C}$ such that $\gamma^{*}\Omega=i_{\tilde{S}}i_{\tilde{X}}(\nu_n)$, $[\tilde{S},\tilde{X}]=\mu.\tilde{X}$, where $\tilde{S}=l_1x_1 \partial/\partial x_1+\cdots+l_n x_n \partial /\partial x_n$. 

From now on we assume that $\beta=\alpha_0$. We claim that if either $d=1$ and $\lambda=0$ or $d \geq 2$ then $\tilde{S}_\gamma=\pm S_{\alpha_0}$, therefore by (3) the proposition follows in both cases. In fact, if $d=1$ and $\lambda=0$, it follows from Proposition \ref{closerlook} (c) that $X=c_1 x_1 \partial/\partial x_1+\cdots+c_n x_n \partial/\partial x_n$, where $c_1,\ldots,c_n \in \mathbb{C}$. Since $\tilde{S}_\gamma$ is tangent to $\mathcal{F}$, there are $a,b \in \mathbb{C}$ such that $\alpha_0^{*}\tilde{S}_\gamma=a.S+b.X$ is a diagonal vector field. For $X$ generic, we have that $a=\pm1$ and $b=0$, then $\tilde{S}_\gamma=\pm S_{\alpha_0}$. If $d \geq 2$, we prove that $\mathcal{G}_S$ is the unique foliation by curves of degree one tangent to $\mathcal{F}$, which implies that $\tilde{S}_\gamma=\pm S_{\alpha_0}$. We assume that $\Omega=i_{R_{n+1}} i_S i_Z (\nu_{n+1})$ is as in the proof of Proposition \ref{splits}. If $Z_1$ is a homogeneous vector field of degree $1$ in $\mathbb{C}^{n+1}$ such that $i_{Z_1}\Omega =0$, there are holomorphic functions $f,g,h$ on $\mathbb{C}^{n+1}$ such that $Z_1=f.R_{n+1}+g.S+h.Z$. Since $i_{R_{n+1}} i_S i_{Z_1} (\nu_{n+1})=h.\Omega$, it follows that $h=0$. In addition $f=f(0)$ and $g=g(0)$ are constant functions. Then the foliation defined in homogeneous coordinates by $Z_1=f(0).R_{n+1}+g(0).S$ is $\mathcal{G}_S$, and the result follows. 

Finally suppose that $d=1$ and $\lambda \neq 0$. As $\lambda+\lambda_1=p_1(d-1)=0$ and $\mathcal{F}(\mathfrak{P},\lambda,2)=\mathcal{F}(\mathfrak{P},-\lambda,2)$, we can assume that $\lambda<0$. It follows from Proposition \ref{closerlook} (c) that $X=\hat{X}_0+\hat{X}_1$, where $\hat{X}_1$ is a linear vector field given by a strictly upper triangular matrix. Since $\tilde{S}_\gamma$ is tangent to $\mathcal{F}$, there are $a,b \in \mathbb{C}$ such that $\alpha_0^{*}\tilde{S}_\gamma=V:=a.S+b.X$. Clearly $a \neq 0$. Note that there is a unique $x_0 \in \mathbb{C}^n$ such that $V(x_0)=0$. One can show that there is an invertible affine map $\psi:\mathbb{C}^n \to \mathbb{C}^n$, $\psi(0)=x_0$, such that $\psi^{*} V=a.S$; hence $a=\pm1$. If $a=1$, as $\omega=i_{V}i_X(\nu_n),[V,X]=\lambda.X$, we can write $\psi^{*}\omega=i_Si_U(\nu_n)$, $[S,U]=\lambda.U$. This implies the existence of $\eta \in \mathcal{C}$ such that $\eta^{*}\Omega=i_Si_U(\nu_n)$, $[S,U]=\lambda.U$ and $\tilde{S}_\gamma=S_{\eta}$. From (3), it follows that $\mu=\lambda$ and $\tilde{S}=S$. If $a=-1$, a similar argument shows that $\mu=\lambda_1=-\lambda$ and $\tilde{S}=S_1$. This concludes the proof of the proposition. 
\end{proof}

Next proposition provides the dimension of a general family $\overline{\mathcal{F}(\mathfrak{P},\lambda,d+1)}$.

\begin{prop}\label{dim}
Assume that $p_1> \cdots>p_n$ are relatively prime positive integers, $\lambda \in \mathbb{Z}$. Then 
\begin{align*}
\dim\left(\overline{\mathcal{F}(\mathfrak{P},\lambda,d+1)}\right)=
\begin{cases}
\dim V_0+n^2+n, d \geq 2, \\
\dim V_0+n^2+n-1, d=1 \text{ and } \lambda\neq0,\\
n^2+2n-2, d=1 \text{ and } \lambda=0.
\end{cases}
\end{align*}
\end{prop}

\begin{proof}
We compute the dimension $k$ of a generic fibre of the map (\ref{Phi}) $$\Phi: V_0 \times \aut(\mathbb{P}^n)  \DashedArrow[->,densely dashed    ] \overline{\mathcal{F}(\mathfrak{P},\lambda,d+1)}.$$ Then $\dim\left(\overline{\mathcal{F}(\mathfrak{P},\lambda,d+1)}\right)=\dim V_0 +\dim \aut(\mathbb{P}^n)-k=\dim V_0+n^2+2n-k$. We use the notation of the previous proposition. Given $\mathcal{F} \in \mathcal{F}(\mathfrak{P},\lambda,d+1)$ defined in homogeneous coordinates by $\Omega$, there is $\beta \in \mathcal{C}$ such that $\beta^{*}\Omega=\omega_{\overline{Y}}=\frac{1}{\tau}i_Si_{\overline{Y}}(\nu_n),[S,\overline{Y}]=\lambda.\overline{Y}$. We can assume that $\beta=\alpha_0$. 

For $Y \in W_0$, let $\Omega_Y$ be the homogeneous ($n-2$)-form of degree $d+2$ satisfying $\alpha_0^{*}\Omega_Y=\omega_Y$. 

\begin{claim}
The linear map $Y \in W_0 \mapsto \Omega_Y$ is injective.
\end{claim}
\begin{proof}
If $\Omega_Y=0$ then $\omega_Y$ also vanishes. So there exists a polynomial $f$ such that $Y=f.S$. Since $[S,Y]=\lambda.Y$, we have that $S(f)=\lambda.f$. This implies that $0=\dive(Y)=\tau.f$ and we get $f=0$. Therefore $Y=0$ and this finishes the proof.
\end{proof}

Note that $([Y],T) \in \Phi^{-1}(\mathcal{F})$ if and only if there exists $\xi \in \mathbb{C}^{*}$ such that $\alpha_0^{*}\left(T^{*}\Omega_Y\right)=\xi. \omega_{\overline{Y}}$. Taking instead of $T \mapsto t.T,t\in \mathbb{C}^{*}$, we can assume that $\xi=1$. Then 
\begin{equation}\label{I}
\alpha_{0}^{*}\Omega_Y=\omega_Y \text{ and }(T \circ \alpha_0)^{*}\Omega_Y=\omega_{\overline{Y}}.   
\end{equation}

If either $d \geq 2$ or $d=1$ and $\lambda=0$, by (\ref{I}) it follows that $S_{T \circ \alpha_0}=\pm S_{\alpha_0}$. By (1) of the previous proposition, if $S_{T \circ \alpha_0}=S_{\alpha_0}$ then $T$ is an element of the subgroup $D(n+1) \subset \mathrm{PGL}(n+1,\mathbb{C})$ of diagonal matrices. On the other hand, if $S_{T \circ \alpha_0}=- S_{\alpha_0}$ then $S=S_1$ and $T$ belongs to the subgroup $A(n+1) \subset \mathrm{PGL}(n+1,\mathbb{C})$ of secondary diagonal matrices. 

Assume that $T \in D(n+1)$ and set $\tilde{T}=\alpha_{0}^{-1}\circ T \circ \alpha_0$. From (\ref{I}) we have that $\tilde{T}^{*}\omega_Y=\omega_{\overline{Y}}$. Since $\tilde{T}^{*}S=S$ it follows from the Claim that $[Y]=[\tilde{T}_{*}\overline{Y}] \in V_0$. Thus $\Phi^{-1}(\mathcal{F})=\{(\tilde{T}_{*}\overline{Y},T) \mid T \in D(n+1)\}$ as long as $\mathfrak{P} \neq \overline{\mathfrak{P}}$ or $\lambda \neq \lambda_1$. If $d \geq 2 $, $\mathfrak{P} = \overline{\mathfrak{P}}$ and $\lambda=\lambda_1$, it follows from a similar argument that $\Phi^{-1}(\mathcal{F})$ has two irreducible components $$\Phi^{-1}(\mathcal{F})=\{(\tilde{T}_{*}\overline{Y},T) \mid T \in D(n+1)\} \cup \{(\hat{T}_{*}\overline{Y}_1,T) \mid T \in A(n+1)\},$$ where $\alpha_1 \in \mathcal{C}$ is given by $\alpha_1(x_1,\ldots,x_n)=(1:x_n:\cdots:x_1)$, $\hat{T}=\alpha_{0}^{-1}\circ T \circ \alpha_1$ and $\overline{Y}_1=\rot(\alpha_1^{*}\Omega)$. In any case $k=\dim D(n+1)=\dim A(n+1)=n$. If $d=1$ and $\lambda=0$, we have that $W_0=\{c_1x_1 \partial/\partial x_1+\cdots+c_nx_n \partial/\partial x_n\mid c_1+\cdots+c_n=0\}$, hence $\dim V_0=n-2$.

Finally, if $d=1$ and $\lambda \neq 0$ then the dimension $k$ of a generic fibre is one more. This implies the proposition.
\end{proof}

\begin{example}
For $S=\overline{\mathcal{F}(4,2,1;3,3)}$, $W_0=\{-2axyz \partial /\partial x+bxz \partial / \partial y+(ayz^2+c_1x+c_2y^2)\partial / \partial z \mid a,b,c_1,c_2 \in \mathbb{C}\}$. Hence $\dim S=3+3^2+3=15$. 
\end{example}

\section{Quasi-homogeneous singularities}
In this section, we recall a recent result concerning stability of quasi-homogeneous singularities.

\begin{definition}\label{dsqh}
Let $\omega$ be a germ of integrable $(n-2)$-form defined at $p \in \mathbb{C}^n$, with $p \in \sing(\mathcal{\omega})$ and $n \geq 3$. We say that $p \in \mathbb{C}^n$ is a \textit{quasi-homogeneous} singularity of $\omega$ if $p$ is an isolated singularity of $Y=\rot(\omega)$ and $DY(p)$ is nilpotent.
\end{definition}

The next theorem was recently proved (\cite{15}). A stronger version for the case $n=3$ was already known (\cite{12}).

\begin{thm}\label{tsqh}
Assume that $0 \in \mathbb{C}^n$ is a quasi-homogeneous singularity of $\omega$. Then there exists a holomorphic coordinate system $w=(w_1,\ldots,w_n)$ around $0 \in \mathbb{C}^n$ where $\omega$ has polynomial coefficients. More precisely, there exist two polynomial vector fields $Z$ and $Y$ in $\mathbb{C}^n$ such that
\begin{enumerate}[(a)]
\item
$Z=S+N$, where $S=\sum_{j=1}^{n}p_j w_j {\partial}/{\partial w_j}$ is linear semi-simple with eigenvalues $p_1,\ldots,p_n \in \mathbb{Z}_{>0}$, $DN(0)$ is linear nilpotent and $[S,N]=0$;
\item
$[N,Y]=0$ and $[S,Y]=\lambda.Y$, where $\lambda \in \mathbb{Z}_{>0}$. In other words, $Y$ is quasi-homogeneous with respect to $S$ with weight $\lambda$;
\item
In this coordinate system we have $\omega=\frac{1}{\lambda+tr(S)}i_Zi_Y (dw_1 \wedge \ldots \wedge dw_n)$ and $L_Y (\omega)=(\lambda+tr(S)) \omega$.
\end{enumerate}
\end{thm}

\begin{definition}\label{type}
In the situation of Theorem \ref{tsqh}, $S=\sum_{j=1}^{n}p_j w_j {\partial}/{\partial w_j}$ and $[S,Y]=\lambda Y$, we say that the quasi-homogeneous singularity is of type $(p_1,\ldots,p_n;\lambda)$.
\end{definition}

We are mainly interested in the following consequence of Theorem \ref{tsqh}.

\begin{cor}\label{cor-qh}
Assume that $\omega=i_{Z}i_{Y}(\nu_n), d\omega=i_{Y}(\nu_n)$, and $0 \in \mathbb{C}^n$ is a quasi-homogeneous singularity of the integrable ($n-2$)-form $\omega$. Then the eigenvalues of $DZ(0)$ are all positive rational numbers.
\end{cor}

The main ingredient of the proof of Theorem \hyperref[b]{B} is Proposition \ref{prop-sing} below, which in turn is based on the following lemma (see Lemma 4.1 of \cite{15}).

\begin{lemma}\label{nilpotent.lemma}
Let $A$ and $L$ be linear vector fields on $\mathbb{C}^n$ such that $\left[L,A \right ] = \mu . A$, where $\mu \neq 0$. Then $A$ is nilpotent.
\end{lemma}

\begin{prop}\label{prop-sing}
Suppose that $\mathcal{F} \in \mathcal{F}(\mathfrak{P},\lambda,d+1)$, where $\lambda \in \mathbb{Z}_{>0}$ and $d \geq 1$.
\begin{enumerate}[(a)]
\item
If the singularity $q_0 \in E_0$ is GK, then it is quasi-homogeneous;
\item
If $q_i \in E_i,i=2,3,\ldots,n$, is a non-Kupka GK singularity, then $\lambda=p_i(d-1)$.
\end{enumerate}
\end{prop}
\begin{proof}
Assume that $\mathcal{F}$ is defined on $E_0$ by $\omega=i_{S}i_{X} (\nu_n)$. As $\lambda>0$ and $[S,Y]=\lambda.Y$, where $Y=\rot(\omega)$, it follows by Remark \ref{remark 1} that $0$ is an isolated singularity of $Y$. Also, since $[S,Y]=\lambda.Y$ we have that $[S,DY(0)]=\lambda. DY(0)$. Then (a) follows from Lemma \ref{nilpotent.lemma} with $L=S, A=DY(0), \mu=\lambda >0$.

For (b), suppose that there is some $i \in \{2,3,\ldots,n\}$ such that $q_i$ is a non-Kupka singularity and $\lambda \neq p_i(d-1)$, i.e., $\lambda_i=\lambda-p_i(d-1)\neq0$. By Proposition \ref{charts} (a), $\omega_i$ defines $\mathcal{F}$ on $E_i$ and $$\tau_i. \omega_i=i_{S_i} i_{Y_i} (\nu_n),[S_i,Y_i]=\lambda_i. Y_i,Y_i=\rot(\omega_i),$$ which implies $[S_i,DY_i (0)]=\lambda_i . DY_i (0)$. It follows from Lemma \ref{nilpotent.lemma}, with $L=S_i,A=DY_i(0),\mu=\lambda_i \neq 0$, that $DY_i(0)$ is nilpotent. 

If $\tau_i \neq 0$, from Corollary \ref{cor-qh} we get a contradiction, since by Proposition \ref{charts} (b) the eigenvalues of $\frac{S_i}{\tau_i}$ are not all positive. If $\tau_i=0$, there exists a polynomial $f$ such that $Y_i=f.S_i$. Set $l=f(0)$. Since $q_i$ is GK, we have that $l \neq 0$. Then we obtain a contradiction since $DY_i(0)$ is nilpotent.
\end{proof}

In the next result (\cite{15}, Theorem 3) we will consider the problem of deformation of two dimensional foliations with a quasi-homogeneous singularity. Consider a holomorphic family of $(n-2)$-forms, $(\omega_t)_{t \in U}$, defined on a polydisc $Q$ of $\mathbb{C}^n$, where the space of parameters $U$ is an open set of $\mathbb{C}^k$ with $0 \in U$. Let us assume that
\begin{itemize}
\item
For each $t \in U$ the form $\omega_t$ defines a two dimensional foliation $\mathcal{F}_t$ on $Q$. Let $\left(Y_t\right)_{t \in U}$ be the family of holomorphic vector fields on $Q$ such that $d\omega_t = i_{Y_t}(\nu_n)$;
\item
$0 \in \mathbb{C}^n$ is a quasi-homogeneous singularity of $\mathcal{F}_0$.
\end{itemize}

\begin{thm}\label{thm-est-sqh}
In the above situation there exist a neighbourhood $0 \in V \subset U$, a polydisc $0 \in P \subset Q$, and a holomorphic map $\mathcal{P}:V \rightarrow P \subset \mathbb{C}^n$ such that $\mathcal{P}(0)=0$ and for any $t \in V$ then $\mathcal{P}(t)$ is the unique quasi-homogeneous singularity of $\mathcal{F}_t$ in $P$. Moreover, $\mathcal{P}(t)$ is of the same type as $\mathcal{P}(0)$, in the sense that if $0$ is a quasi-homogeneous singularity of type $(p_1,\ldots,p_n;\lambda)$ of $\mathcal{F}_0$ then $\mathcal{P}(t)$ is a quasi-homogeneous singularity of type $(p_1,\ldots,p_n;\lambda)$ of $\mathcal{F}_t, \forall t \in V$. 
\end{thm}
\begin{proof}[Proof of Theorem A] 
Let $\mathcal{F} \in \mathcal{F}(\mathfrak{P},\lambda,d+1)$ be the required WGK foliation. By Proposition \ref{prop-sing} (a), $q_0$ is a quasi-homogeneous singularity of $\mathcal{F}$. 

Let $(\mathcal{F}_t)_{t \in \Sigma}$ be a holomorphic family of foliations in $\mathscr{F}_2(d+1,n)$, parameterized in a open set $0 \in \Sigma \subset \mathbb{C}$, where $\mathcal{F}_0=\mathcal{F}$, and $(\Omega_t)_{t \in \Sigma}$ a holomorphic family of respective homogeneous $(n-2)$-form on $\mathbb{C}^{n+1}$ that defines $\mathcal{F}_t$. It suffices to prove that $\mathcal{F}_t \in \mathcal{F}(\mathfrak{P},\lambda,d+1)$ for small $\abs{t}$.

Next we show that $\mathcal{F}_t$ is WGK for small $\abs{t}$. Define $\omega_{i,t}=\Omega_t|_{E_i},i=0,\ldots,n$, where $E_0,\ldots,E_n$ are defined as in Proposition \ref{charts}. Set $$\mathcal{S}_{i,t}=\{[z] \in E_i \mid \omega_{i,t}(z)=0\} \text{ and } \mathcal{T}_{i,t}=\{[z] \in E_i \mid d\omega_{i,t}(z)=0\}.$$ 

Denote by $\mathcal{Q}_{i,t}$ and $\mathcal{R}_{i,t}$ the union of the components of codimension $\geq 3$ and the union of the components of codimension $\leq 2$ of $\mathcal{T}_{i,t}$, respectively. By definition, $\mathcal{F}_t$ is WGK on $E_i$ means that $\mathcal{S}_{i,t} \cap \mathcal{R}_{i,t} = \emptyset$. 

For each $p \in \mathbb{P}^n$, take an open set $V_p \subset \mathbb{P}^n$ with compact closure such that $p \in V_p \subset \overline{V_p} \subset E_i$, for some $i=i(p) \in \{0,\ldots,n\}$. As $\mathcal{F}_0$ is WGK, there exists $\epsilon_p >0$ such that $\mathcal{S}_{i,t} \cap \mathcal{R}_{i,t} \cap \overline{V_p} = \emptyset$ if $\abs{t} < \epsilon_p$. By the compactness of $\mathbb{P}^n$, we can assume that there exist a finite number of points $p_1,\ldots,p_m$ such that $$\mathbb{P}^n=\bigcup_{j=1}^{m} V_{p_j}.$$ Then $\mathcal{F}_t$ is WGK, if $\abs{t} < \epsilon$, where $\epsilon = \displaystyle\min_{j\in \{1,\ldots,m\}} \epsilon_{p_j}$.

Hereafter, the proof of Theorem \hyperref[a]{A} is close to that of Theorem \ref{thefour} that can be found in \cite{1}, if we take into account the following three observations.
\begin{enumerate}
\item
As in the case of GK singularities, if $p_0$ is a WGK singularity of a germ of foliation $\mathcal{G}$ defined by the integrable ($n-2$)-form $\eta$, the sheaf of germs of vector fields at $p_0$ tangent to $\mathcal{G}$ is locally free and has two generators. Indeed, let $z=(z_1,\ldots,z_n)$, $z(p_0)=0$, be a coordinate system around $p_0$, and $Y=\rot(\eta)$. It suffices to show that there exists a holomorphic vector field $X$ such that $\eta=i_Xi_Y (dz_1 \wedge \cdots \wedge dz_n)$ (see the proof of Corollary 2 of \cite{1}). On the other hand, as $\cod(\sing(Y)) \geq 3$, the latter is a consequence of Proposition 1 of \cite{15}.
\item The tangent sheaf of $\mathcal{F}$ splits as $\mathcal{T}\mathcal{F}=\mathcal{O} \oplus \mathcal{O}(1-d)$.
\item
In the proof of Theorem \hyperref[a]{A}, Theorem \ref{thm-est-sqh} plays the same role that Proposition 1 of \cite{1} does in the proof of Theorem \ref{thefour}.
\end{enumerate}

\end{proof}

\section{Depicting GK components}
In this section we prove Theorems \hyperref[b]{B}, \hyperref[thmB1]{B.1} and \hyperref[thmB2]{B.2}.
\begin{proof}[Proof of Theorem B]
If $\overline{\mathcal{F}(\mathfrak{L},\mu,d+1)}$ is a GK component of $\mathscr{F}_2(d+1,n)$, we begin by showing that there exist $p_1,\ldots,p_n,\lambda$ satisfying the conditions of Theorem \hyperref[b]{B} such that $\overline{\mathcal{F}(\mathfrak{L},\mu,d+1)}=\overline{\mathcal{F}(\mathfrak{P},\lambda,d+1)}$. In this case, by Proposition \ref{symmetry}, either $\mathfrak{L}=\mathfrak{P}$, $\mu=\lambda$ or $\mathfrak{L}=\overline{\mathfrak{P}}$, $\mu=\lambda_1$.

The main idea of the proof is to look at the singularities of $S$. This is obviously observed in part (a) of Theorem \hyperref[b]{B}. The part (b) of the theorem relates to the fact that the singularities $q_2,q_3,\ldots,q_n$ are also GK singularities of $\mathcal{F} \in \mathcal{F}(\mathfrak{P},\lambda,d+1)$.

We consider that a GK foliation $\mathcal{F} \in \mathcal{F}(\mathfrak{P},\lambda,d+1)$ is equipped with the parameters of Proposition \ref{charts}. By Proposition \ref{prop-sing} (b), if $d \geq 2$, as $p_1,\ldots,p_n$ are pairwise distinct, the singularities $q_2,q_3,\ldots,q_n$ are Kupka, with at most one exception. If $d=1$, since $\lambda_i=\lambda-p_i(d-1)=\lambda \neq 0$, $q_2,\ldots,q_n$ are Kupka singularities of $\mathcal{F}$ (the case where $d=1$ and $\lambda=0$ will be treated in Corollary \ref{gen-klein-lie}). 

We have the following useful observations easily verified by the reader.
\begin{enumerate}[i)]
\item
For each $i \in \{2,\ldots,n\}$, if the $j$-th entry of $Y_i(0)$ is not $0$, then it follows from $[S_i,Y_i]=\lambda_i.Y_i$ and Proposition \ref{closerlook} (c) that condition $c_{i-1j}$ is satisfied.
\item
If $c_{ij}$ and $c_{i_1 j_1}$ hold at the same time, $1 \leq i,i_1 \leq n-1$, $1 \leq j,j_1 \leq n$, then $i_1>i$ implies that $j_1 > j$.
\item
Denote by $\overline{c}_{ij}$ the same condition as $c_{ij}$ substituting $p_k$ by $\overline{p}_k$ and $\lambda$ by $\lambda_1,i=1,\ldots,n-1,j=1,\ldots,n$. It follows that 
\begin{center}
$c_{ij}$ holds $\iff \overline{c}_{n-i,n-j+1}$ holds, $i=1,\ldots,n-1,j=1,\ldots,n$.
\end{center}
\item
Assume that $\lambda=p_i(d-1)$, for some $i \in \{2,\ldots,n\}$. If $c_{i_1j_1}$ holds, it follows that $i+1 \leq j_1  \leq n$, if $i_1 > i-1$, and $1 \leq j_1 \leq i-2$, if $i_1 <i-1$.
\end{enumerate}

It is clear that the condition of Theorem \hyperref[b]{B} (a) must be satisfied. Suppose first that the singularities $q_2,\ldots,q_n$ are all of Kupka type. Thanks to i) and ii), $p_1,\ldots,p_n,\lambda,d$ must satisfy the conditions $$c_{11},c_{22},\ldots,c_{ii},c_{i+1,i+2},c_{i+2,i+3},\ldots,c_{n-1,n},$$ where $0 \leq i \leq n-1$. From Proposition \ref{symmetry} and iii), we can assume that $0 \leq i \leq \left\lfloor \frac{n-1}{2} \right\rfloor$. In addition, for each $j \in \{2,\ldots,n\}$, since $\tau_j.\omega_j=i_{S_j}i_{Y_j}(\nu_n)$ and $Y_j(0) \neq 0$, we have that $\tau_j \neq 0$. Thus we are in the situation of Theorem \hyperref[b]{B} (b.1).

Now suppose that the singularities $q_2,\ldots,q_n$ are of Kupka type, except $q_i$. Thanks to i), ii), iv) and Proposition \ref{prop-sing} (b), $p_1,\ldots,p_n,\lambda,d$ must satisfy $$c_{11},c_{22},\ldots,c_{i-2,i-2},\lambda=p_i(d-1),c_{i,i+1},c_{i+1,i+2},\ldots,c_{n-1,n}.$$ From Proposition \ref{symmetry}, iii) and the equivalence $\lambda=p_i(d-1) \iff \lambda_1=\overline{p}_{n+2-i} (d-1)$, we can assume that $2 \leq i \leq \lfloor \frac{n+2}{2} \rfloor$. As in the previous case we have $\tau_j \neq 0, j \in \{2,\ldots,n\} \setminus \{i\}$. Thus we are in the situation of Theorem \hyperref[b]{B} (b.2).

Therefore the conditions of Theorem \hyperref[b]{B} are needed to the existence of GK foliations in $\mathcal{F}(\mathfrak{P},\lambda,d+1)$. Next we show that the conditions of Theorem \hyperref[b]{B} are also sufficient. The proof follows immediately from the next two lemmas.

\begin{lemma}\label{reduction}
A foliation $\mathcal{F} \in \mathcal{F}(\mathfrak{P},\lambda,d+1)$ is GK if and only if the singularities $q_0,q_2,q_3,\ldots,q_n$ of $\mathcal{F}$ are GK.
\end{lemma} 
\begin{proof}
Of course, if $\mathcal{F}$ is GK then the singularities $q_0,q_2,\ldots,q_n$ are GK. Conversely, assume that they are GK singularities of $\mathcal{F}$. Suppose that there exists a singularity $p$ that is not GK.

Assume that $p \neq q_1$. The orbit of the global vector field $S$ through any point $z \not\in \sing(S)$ accumulates at two points of $\sing(S)$, say $q_i$, $q_j$, $i \neq j$ and $i \neq 1$. Since $[S_i,Y_i]=\lambda_i .Y_i$, it follows from Proposition \ref{prop-qh-vf} (b) that the orbit of $S$ through $p$ is contained in $\sing(Y_i)$. We obtain a contradiction, since $q_i$ is GK. 

Next, suppose that $p=q_1$. It is not difficult to see that there is a non-GK singularity of $\mathcal{F}$ on $E_1$ other than $q_1$. Once again this contradicts $q_0,q_2,\ldots,q_n$ being GK. 
\end{proof}

Recall the parameters $W_0$, $V_0$, $\omega_Y$ introduced before Proposition \ref{irreducibility}. For $Y \in W_0$, set $Y_i,i=1,\ldots,n$, as in Proposition \ref{charts} (computed in the same way as $\omega_0=\omega_Y$ provides a foliation of degree $d+1$).

\begin{lemma}
Under any situation of Theorem \hyperref[b]{B}, there exists a proper algebraic subset $\Delta \subset V_0$ such that if $[Y] \in V_0 \setminus \Delta$, then the singularities $q_0,q_2,q_3,\ldots,q_n$ of the foliation $\mathcal{F}(S,Y) \in \mathcal{F}(\mathfrak{P},\lambda,d+1)$ defined by $\omega_Y$ on $E_0$ are GK. 
\end{lemma}
\begin{proof}
Consider the following subsets of $V_0$, where $i$ is given by Theorem B.
\begin{align*}
\Gamma&=\{[Y] \in V_0 \mid \omega_Y \text{ does not define a foliation of degree } d+1 \text{ on } \mathbb{P}^n \},\\ 
\Sigma&= \{ [Y] \in V_0 \mid q_0 \text{ is a non-isolated singularity of } Y \}, \\
L_i&=\{[Y] \in V_0 \mid \det(DY_i(q_i))=0\},\\
H_j&=\{[Y] \in V_0 \mid Y_j(q_j)=0\},j=2,3,\ldots,n. 
\end{align*}

We proceed in the following way. We start by observing that $\Gamma$, $\Sigma$, $L_i$ and the $H_j$'s are algebraic subsets of $V_0$. Note that if $Y_i(q_i)=0$ and $[Y] \notin L_i$, then $q_i$ is an isolated singularity of $Y_i$. Under any situation of Theorem \hyperref[b]{B}, $\Gamma$ and $\Sigma$ are proper subsets of $V_0$. In the situation of Theorem \hyperref[b]{B} (b.2), we show additionally that $L_i$, $H_j, j \in \{2,3,\ldots,n\} \setminus \{i\}$, are proper subsets of $V_0$, so we can take $$\Delta=\Gamma \cup \Sigma \cup L_i \cup \bigcup_{\overset{k=2}{k \neq i}}^{n} H_k.$$ Analogously, in the situation of Theorem \hyperref[b]{B} (b.1), the $H_j$'s are proper subsets of $V_0$, then we take $$\Delta=\Gamma \cup \Sigma \cup \bigcup_{k=2}^{n} H_k.$$

It is easy to see that $\Gamma$, $L_i$ and the $H_j$'s are algebraic. For $\Sigma$, by Proposition \ref{closerlook} (c) the change $(x_1,\ldots,x_n) \mapsto (x_1^{p_1},\ldots,x_n^{p_n})$ turn the entries of $Y \in W_0$ into homogeneous polynomials, thus we can use the multipolynomial resultant for $n$ homogeneous polynomials to conclude that $\Sigma$ is algebraic (see \cite{14} for example).

Since the condition $c_{n-1,n}$ holds in any situation of Theorem \hyperref[b]{B}, we have $\lambda=p_n d>0$. It follows from Remark \ref{remark 1} and Theorem \hyperref[b]{B} (a) that $\Sigma \subset V_0$ is proper. We claim that $\Gamma \subset V_0$ is also proper. In fact, with exception to very few cases, the condition $c_{n-2,n-1}$ is also satisfied, i.e., $p_n+\lambda=p_{n-1}d$. As also $\lambda=p_n d$, we have $$\overline{X}=x_{n}^{d} \cdot R_n+x_{n-1}^{d}\frac{\partial}{\partial x_n}$$ is such that $[S,\overline{X}]=\lambda.\overline{X}$. Then the singular set of $\overline{\omega}=i_S i_{\overline{X}}  (\nu_n)$ has no divisorial components, and $[\overline{Y}] \notin \Gamma$, where $\overline{Y}=\rot(\overline{\omega})$. Likewise one can check that $\Gamma$ is proper in the other cases.

Next assume that we are in the situation of Theorem \hyperref[b]{B} (b.2). We show that $L_i \subset V_0$ is proper. In fact, let $\omega_i=i_{S_i}i_{X_i}(\nu_n)$, $[S_i,X_i]=\lambda_i .X_i=0$, like in Proposition \ref{charts} (a) and defining some foliation of $\mathcal{F}(\mathfrak{P},\lambda,d+1)$ on $E_i$. By the parametrization $\Phi$ (\ref{Phi}), if $Y_i=\rot(\omega_i)=\tau_i.X_i-\dive(X_i).S_i$ is such that $\det(DY_i(q_i)) \neq 0$, then we are done. Otherwise, let $\epsilon_1,\epsilon_2,\ldots,\epsilon_n$ denote $n$ arbitrary non-real complex numbers satisfying $$\sum_{k=1}^{n} \epsilon_k \in \mathbb{Z}-\{0\}.$$ 

Set $\tilde{\omega}_i=i_{S_i}i_{\tilde{X}_i}(\nu_n)$, where $\tilde{X}_i=X_i+\epsilon(\sum_{k=1}^n \epsilon_k x_k \frac{\partial}{\partial x_k}), \epsilon \in \mathbb{C}$. We have that $[S_i,\tilde{X}_i]=0$ and $\tilde{Y}_i=\rot(\tilde{\omega}_i)=Y_i+\epsilon.Z$, where $$Z=\tau_i .\left(\sum_{k=1}^{n} \epsilon_k x_k \frac{\partial}{\partial x_k}\right)- \left(\sum_{k=1}^n \epsilon_k\right).S_i.$$

Since $Z$ is a diagonal vector field satisfying $\det(DZ(q_i)) \neq 0$, if we take $\abs{\epsilon}$ sufficiently large we have $\det(D\tilde{Y}_i(q_i)) \neq 0$. This finishes the proof that $L_i \subset V_0$ is proper. 

Finally, in both situations of Theorem \hyperref[b]{B}, $H_j \subset V_0$ is proper, $j=2,\ldots,n$. In fact, let $\omega_j=i_{S_j}i_{X_j}(\nu_n)$, $[S_j,X_j]=\lambda_j .X_j$, defining some foliation of $\mathcal{F}(\mathfrak{P},\lambda,d+1)$ on $E_j$. As $\tau_j \neq 0$ and either $c_{j-1,j-1}$ or $c_{j-1,j}$ is verified, if necessary we can redefine $\omega_j$ by adding to $X_j$ either $c \cdot \partial / \partial x_j$ in the former case or $c \cdot \partial / \partial x_{j+1}$ in the latter case, where $c \in \mathbb{C}^{*}$, in order to obtain $Y_j(q_j) \neq 0$. Once again by the parametrization $\Phi$ it is sufficient to conclude that $H_j \subset V_0$ is proper.
\end{proof}

\end{proof}

\begin{remark}\label{division}
In any situation of Theorem \hyperref[b]{B}, there is some $k \in \{1,\ldots,n\}$ such that $p_1$ divides $p_k+\lambda$. In fact, if $0$ is an isolated singularity of some $Y \in W_0$, there must exist $m \in \mathbb{Z}_{>0}$ and $k \in \{1,\ldots,n\}$ such that $x_{1}^{m} \partial / \partial x_k \in W_0$, for otherwise we would have $\{(x_1,\ldots,x_n) \mid x_2=x_3=\cdots=x_n=0\} \subset \sing(Y)$, a contradiction. By Proposition \ref{closerlook} (c), $m \cdot p_1=p_k+\lambda$, which implies the result. 
\end{remark}

\begin{remark}\label{dplus1term}
If $\lambda=p_n d$, one can show that for $Y \in W_0$ there is $\mu \in \mathbb{C}$ such that $\hat{Y}_{d+1}=\mu. x_{n}^d (\tau.R_n-(n+d).S)=\mu.x_{n}^d\left(-\tau_1 x_1 \partial/\partial x_1+\sum_{k=2}^{n} \tau_k x_k \partial / \partial x_k\right)$. This happens in Theorem \hyperref[b]{B}, since the condition $c_{n-1,n}$ is always valid.
\end{remark}

\begin{cor}\label{gen-klein-lie}
Let $\mathfrak{L}=(d^{n-1}+\cdots+1,d^{n-2}+\cdots+1,\ldots,d+1,1)$. Then for every $d \geq 1$, $\overline{\mathcal{F}(\mathfrak{L},-1,d+1)}$ is an irreducible component of $\mathscr{F}_2(d+1,n)$ of dimension $k(n,d)$, where $k(n,d)=n^2+2n-1$ if $d\geq 2$ and $k(n,1)=n^2+2n-2$. This component is the closure of a $\mathrm{PGL}(n+1,\mathbb{C})$ orbit on $\mathscr{F}_2(d+1,n)$. Furthermore, if $d=1$ this is the unique GK component of $\mathscr{F}_2(2,n)$ of the form $\overline{\mathcal{F}(\mathfrak{P},\lambda,2)}$.
\end{cor}
\begin{proof}
For $i=1,\ldots,n$, set $r_i=\sum_{j=i-1}^{n-1}d^j$, $\mu=d^{n-1}$ and $\mathfrak{R}=(r_1,\ldots,r_n)$. We show first that $\mathcal{F}(\mathfrak{R},\mu,d+1)$ satisfy the conditions of Theorem \hyperref[b]{B}. By Proposition \ref{symmetry} we have that $\mathcal{F}(\mathfrak{R},\mu,d+1)=\mathcal{F}(\mathfrak{L},-1,d+1)$, and the result follows. It is straightforward the verification that the conditions $c_{12},c_{23},\ldots,c_{n-1,n}$ in Theorem \hyperref[b]{B} (b.1), $i=0$, are verified, i.e., for $p_1=r_1, \ldots,p_n=r_n,\lambda=\mu$ we have 
\begin{equation}\label{sitB}
p_3+\lambda=p_2 d,p_4+\lambda=p_3d,\ldots,p_n+\lambda=p_{n-1}d,\lambda=p_nd.
\end{equation}

Since $r_2,\ldots,r_n$ are multiple of $d$ and $r_1$ is not, we have that $\tau_j =\tau-r_j(n+d)\neq 0,j=2,\ldots,n$. We assert that $$W_0=\{ \mu.x_n^d \left(\tau . R_n-(n+d).S \right)+\sum_{k=2}^n a_{k-1}.x_{k-1}^d \frac{\partial}{\partial x_k} ; \mu,a_1,\ldots,a_{n-1} \in \mathbb{C}  \}.$$

In fact, by Proposition \ref{closerlook} (c), it is due to Remark \ref{dplus1term} and the following result.

\begin{claim}
Given $k \in \{1,\ldots,n\}$, the system $b_1\cdot r_1 +\cdots+b_n\cdot r_n=r_k+\lambda$, $b_1+\cdots+b_n \leq d$ and $b_1,\ldots,b_n \in \mathbb{Z}_{\geq 0}$ has no solution if $k=1$. If $k \neq 1$, then $b_j=0,j \neq k-1$, $b_{k-1}=d$ is the unique solution.
\end{claim}
\begin{proof}
The above equality means that 
\begin{equation}\label{basis}
\sum_{i=0}^{n-1}\left(\sum_{j=1}^{i+1}b_j\right) d^i=d^n+d^{n-1}+\cdots+d^{k-1}.
\end{equation}
If $b_1+\cdots+b_n < d$, we have that both sides of (\ref{basis}) provide the representation of $r_k+\lambda$ in the base $d$ system, and since the term $d^n$ would not appear on the left-hand side we obtain a contradiction. Assume that $b_1+\cdots+b_n=d$. If $k=1$, from (\ref{basis}) we have $\sum_{j=1}^{k}b_j=d$, for any $k=2,\ldots,n$ and $b_1=d+1$, and we get a contradiction. If $k>1$, a similar argument shows that $b_{k-1}=d$ and $b_j=0,j \neq k-1$. 
\end{proof}

Setting $\mu=a_1=\cdots=a_{n-1}=1$ in the definition of $W_0$, by Remark \ref{dplus1term} it follows that $0$ is an isolated singularity of $$Y=-\tau_1 x_1x_{n}^d \partial/\partial x_1+\sum_{k=2}^{n} (\tau_k x_k x_{n}^d + x_{k-1}^d )\frac{\partial}{\partial x_k},$$ then Theorem \hyperref[b]{B} (a) holds (recall that $\tau_1 \neq 0$). Moreover, since $\dim V_0=n-1$, it follows from Proposition \ref{dim} that $\overline{\mathcal{F}(\mathfrak{R},\mu,d+1)}$ has dimension $k(n,d)$. Of course the action of $\aut(\mathbb{P}^n)=\mathrm{PGL}(n+1,\mathbb{C})$ on $\mathscr{F}_2(d+1,n)$ leaves $\mathcal{F}(\mathfrak{L},-1,d+1)$ invariant, and by the description of $W_0$ it is easy to see that this action is transitive.

Next, we show that of all components $\overline{\mathcal{F}(\mathfrak{P},\lambda,d+1)}$ described by Theorem \hyperref[b]{B}, $\overline{\mathcal{F}(\mathfrak{R},\mu,d+1)}$ is the only one where $\lambda_1 <0$. In fact, with exception to the situation of Theorem \hyperref[b]{B} (b.1), $i=0$, the condition $c_{11}$ holds, i.e., $p_1+\lambda=p_2d$, which implies $\lambda_1=\overline{p}_{n}d>0$. In the situation of Theorem \hyperref[b]{B} (b.1), $i=0$, $p_1,\ldots,p_n,\lambda,d$ satisfy (\ref{sitB}) and there exists a positive integer $m$ such that $$p_2=m r_2,p_3=m r_3,\ldots,p_n=m r_n,\lambda=m d^{n-1}.$$ 

Assume, by Remark \ref{division}, that $p_1$ divides $p_2+\lambda$. This means that $p_1$ divides $d^{n-1}+\cdots+d+1$. Since $p_1>p_2$, we have that $m \leq d$ and
\begin{equation*}
\lambda_1 =p_1(d-1)-\lambda \geq (p_2+1)(d-1)-md^{n-1}=d-(m+1) \geq -1.
\end{equation*}  

Moreover, if $\lambda_1=-1$ then $m=d$ and $p_1 = d^{n-1}+d^{n-2}+\cdots+d+1$, and we obtain $\overline{\mathcal{F}(\mathfrak{R},\mu,d+1)}$. In an analogous way, one can show that if $p_1$ divides $p_l+\lambda$, for $l \neq 2$, then $\lambda_1 \geq 0$.

Finally we consider $d=1$. In this case $\lambda+\lambda_1=p_1(d-1)=0$. Suppose first that $\lambda,\lambda_1 \neq 0$. By Proposition \ref{symmetry}, we can assume that $\lambda_1<0$, and we obtain $\overline{\mathcal{F}(\mathfrak{R},\mu,d+1)}$ with $d=1$. Now suppose that $\lambda=\lambda_1=0$. Take some $\mathcal{F} \in \mathcal{F}(\mathfrak{P},0,2)$, defined by $\omega=\frac{1}{\tau}i_S i_Y (\nu_n)$ on $E_0$, $Y=\rot(\omega)$. Since $q_0$ is is an isolated singularity of $Y$, there are $c_1,\ldots,c_n \in \mathbb{C}^{*}$ such that $Y=\sum_{k=1}^n c_k x_k \frac{\partial}{\partial x_k}$. Thus $\eta=\frac{\omega}{x_1 \cdots x_n}$ is a logarithmic form defining $\mathcal{F}$, one can deform $\mathcal{F}$ to a logarithmic foliation that does not belong to $\mathcal{F}(\mathfrak{P},0,2)$ and consequently $\overline{\mathcal{F}(\mathfrak{P},0,2)}$ is not an irreducible component of $\mathscr{F}_2(2,n)$. 
\begin{remark}
As a consequence of the proof of Corollary \ref{gen-klein-lie}, $\overline{\mathcal{F}(\mathfrak{L},-1,d+1)}$, $d \geq 1$, are the only irreducible components of the form $\overline{\mathcal{F}(\mathfrak{P},\lambda,d+1)}$ whose generic element has exactly one non-Kupka singularity (quasi-homogeneous), namely $q_1$.
\end{remark}
\end{proof}

\begin{proof}[Proof of Theorem B.2]
We will give an explicit form to the components described by Theorem \hyperref[b]{B}. In the case $n=4$, we have four different situations: Theorem \hyperref[b]{B} (b.1), $i=0$ and $i=1$, Theorem \hyperref[b]{B} (b.2), $i=2$ and $i=3$. These, combined with the possibilities of Remark \ref{division}, give rise to what we call cases. There are, therefore, 16 cases to consider. We begin by showing the three cases where the families $\mathcal{F}(p,q,r,s;\lambda,d+1)$ never contain GK foliations.

In the sequel, given two integers $a$ and $b$, when $a$ divides $b$ we sometimes denote this by $a \mid b$. We also define $$m_1=\frac{(p+\lambda)(q+\lambda)(r+\lambda)(s+\lambda)}{pqrs}.$$
 
\begin{enumerate}[leftmargin=*]
\item Theorem \hyperref[b]{B} (b.1), $i=0$, $p$ divides $r+\lambda$

The conditions $c_{12}$, $c_{23}$ and $c_{34}$ are satisfied, i.e., $r+\lambda=qd$, $s+\lambda=rd$, $\lambda=sd$. An easy verification shows that there exists $m \in \mathbb{Z}_{>0}$ such that 
\begin{equation}\label{eqcase1}
p > q=m(d^2+d+1) > r=m(d^2+d) > s=md^2,\lambda=md^3. 
\end{equation}

Hence $\gcd(p,q,r,s)=1 \iff \gcd(p,m)=1$. As $\lambda>0$, by Remark \ref{remark 1} and Proposition \ref{prop-qh-vf} (d) it suffices to show that $m_1 \not\in \mathbb{Z}$. Suppose, by contradiction, that $m_1 \in \mathbb{Z}$. By (\ref{eqcase1}), this means that $p \mid d^3(d^3+d^2+d+1)$. On the other hand, since $p \mid r+\lambda$ and $p>q$, we have that $\gcd(p,d^2+d+1) \neq 1$. Clearly a common prime factor of $p$ and $d^2+d+1$ cannot divide neither $d^3$ nor $d^3+d^2+d+1$, which is a contradiction. \\

\item Theorem \hyperref[b]{B} (b.1), $i=0$, $p$ divides $s+\lambda$

Once again $p,q,r,s,\lambda$ are given as in (\ref{eqcase1}). We cannot proceed as before, because now $m_1 \in \mathbb{Z}$. Let us write $$Y=A_1 \frac{\partial}{\partial x} + A_2 \frac{\partial}{\partial y} + A_3 \frac{\partial}{\partial z}+A_4 \frac{\partial}{\partial w},$$ for $Y \in W_0$. We claim that $A_1(x,y,z,0) \equiv 0$ and $A_2(x,y,z,0) \equiv 0$, which clearly implies that $0$ is a non-isolated singularity of $Y$. 

We check that $A_1(x,y,z,0) \equiv 0$. Suppose this is not true. Then a monomial term $x^a y^b z^c$ must appear in the expansion of $A_1$. By Proposition \ref{closerlook} (c), $p+\lambda=ap+bq+rc$, that is
\begin{equation}\label{eqcase12}
p(a-1)=m(d^3-b(d^2+d+1)-c(d^2+d)). 
\end{equation}

As $p \mid s+\lambda$, we have that $p \mid d^3+d^2=d^2(d+1)$, so we can write $p=j_1 j_2$, where $j_1 \mid d^2$ and $j_2 \mid d+1$. Since $j_1$ divides the right-hand side of (\ref{eqcase12}), $\gcd(j_1,m)=1$ and $j_1$ divides $d^2$, it follows that 
\begin{align*}
&j_1 \mid b(d+1)+cd=d(b+c)+b \implies j_1 \mid d(d(b+c)+b)\implies \\ &j_1 \mid bd \implies j_1 \mid (d(b+c)+b)-bd=cd+b \implies \\ &j_1 \mid b^2=b(cd+b)-cbd  \implies j_1 \mid d^2-b^2=(d-b)(d+b).
\end{align*}

Since $j_2$ divides the right-hand side of (\ref{eqcase12}), $\gcd(j_2,m)=1$ and $j_2$ divides $d+1$, it follows that
\begin{align*}
&j_2 \mid d^3-b=d^3+1-(b+1) \implies j_2 \mid b+1 \implies \\ &j_2 \mid d-b=(d+1)-(b+1).
\end{align*} 

As $\gcd(j_1,j_2)=1$ and both $j_1,j_2$ divide $d^2-b^2=(d-b)(d+b)$, we conclude that $p=j_1 j_2 \mid d^2-b^2$. Since $b \leq d$, we have $p \leq d^2-b^2 \leq d^2$. As $p>q$, we obtain a contradiction. Therefore $A_1(x,y,z,0) \equiv 0$.

Analogously, if we suppose $A_2(x,y,z,0) \not\equiv 0$, writing $p=j_1 j_2$ as above, it is possible to show that $j_1 \mid d^2-(b-1)^2$ and $j_2 \mid d+1-b$, which implies $p \mid d^2-(b-1)^2$. Once again since $p>q$ we obtain a contradiction. \\

\item Theorem \hyperref[b]{B} (b.2), $i=2$, $p$ divides $s+\lambda$

In this case, $\lambda=q(d-1)$, $c_{23}$ and $c_{34}$ are satisfied, i.e., $s+\lambda=rd$, $\lambda=q(d-1)=sd$. So
\begin{equation}\label{eqcase3}
p > q=md^2> r=m(d^2-1) > s=m(d^2-d),\lambda=m(d^3-d^2), 
\end{equation}
for some $m \in \mathbb{Z}_{>0}$. Hence $\gcd(p,q,r,s)=1 \iff \gcd(p,m)=1$. 

We show that $m_1 \not\in \mathbb{Z}$. In fact, if $m_1 \in \mathbb{Z}$ by (\ref{eqcase3}) $p \mid d^3(d-1)(d^2+d+1)$. On the other hand, $p \mid s+\lambda$ means that $p \mid d^3-d=d(d+1)(d-1)$. Then actually $p \mid d(d-1)(d^2+d+1)=d^4-d$. But then $$p \mid d^4-d - d(d^3-d)=d^2-d<q,$$ and we obtain a contradiction. 
\end{enumerate}

It remains to consider the cases which provide GK irreducible components, corresponding to the situations of Theorem \hyperref[thmB2]{B.2}. In all cases, $p,q,r,s,\lambda,d$ satisfy certain $c_{ij}$'s and we show, in a very similar way, that the other conditions of Theorem B also hold. Therefore we do so only for two cases, which contain the main aspects.

\begin{enumerate}[leftmargin=*]
\setcounter{enumi}{3}
\item Theorem \hyperref[b]{B} (b.1), $i=1$, $p$ divides $s+\lambda$

The conditions $c_{11}$, $c_{23}$ and $c_{34}$ are satisfied, i.e., $p+\lambda=qd$, $s+\lambda=rd$, $\lambda=sd$. Thus 
\begin{equation*}
p=kd > q=md+k > r=m(d+1) > s=md,\lambda=md^2, 
\end{equation*}
for some $m \in \mathbb{Z}_{>0}$. Hence $\gcd(p,q,r,s)=1 \iff \gcd(k,m)=1$ and $p$ divides $s+\lambda$ means that $k$ divides $d+1$.\\

$\bullet$ $\tau_2,\tau_3,\tau_4 \neq 0$ \

In this case, $\tau_2=r+s-3q<0$ and $\tau_4=p+q+r-3s>0$. Suppose that $\tau_3=p+q-3r=0$; this implies that $k(d+1)=m(2d+3)$. Since the pairs $k,m$ and $d+1,2d+3$ are relatively prime, it follows that $m=d+1$ and $k=2d+3$. We obtain a contradiction, since $k$ divides $d+1$.\\

$\bullet$ Theorem \hyperref[b]{B} (a) holds true \

We claim that Theorem \hyperref[b]{B} (a) is satisfied if and only if 
\begin{equation}\label{weirdcond}
\gcd\left(\frac{m(d+1)}{k},d \right)=1. 
\end{equation}
Note that we are in the situation of Theorem \hyperref[thmB2]{B.2} (b), where additionally $k$ divides $d+1$.

Assume that Theorem B (a) holds, i.e, $0$ is an isolated singularity of some $Y \in W_0$. Write $Y \in W_0$ as in case (2). We claim that $A_1(x,0,z,0) \equiv 0$ and $A_3(x,0,z,0) \equiv 0$. Let us check that $A_1(x,0,z,0) \equiv 0$. If it is not true, then a monomial term $x^a z^b$ must appear in the expansion of $A_1$. It follows that $p+\lambda=ap+br$, equivalently, $kd(a-1)=m(d^2-b(d+1))$.

Hence $k$ divides $d^2-b(d+1)$, which implies that $k=1$ since also $k$ divides $d+1$. We get $p=d < q=md+k$, which is a contradiction. By proceeding in an analogous way, we obtain $A_3(x,0,z,0) \equiv 0$. 

As both $A_1(x,0,z,0)$ and $A_3(x,0,z,0)$ vanish, it is necessary that a monomial term $x^a z^b$ appears in the expansion of $A_2$. Thus $q+\lambda=ap+br$, that is, $ad-1=mj(d-b)$, where $j=\frac{d+1}{k} \in \mathbb{Z}_{>0}$. Hence $\gcd(mj,d)=1$ and we have (\ref{weirdcond}).

Conversely, assume that (\ref{weirdcond}) holds and set $j$ as above. As $\gcd(mj,d)=1$, there exists a integer $b$ such that $d \mid mjb-1$. We can assume that $0 < b < d$. Thus $$d \mid mjd-(mjb-1)=mj(d-b)+1.$$ 

If we define $a=\frac{mj(d-b)+1}{d} \in \mathbb{Z}_{>0}$, then $q+\lambda=ap+br$. One can check that $a+b \leq d$. Set $l=\frac{s+\lambda}{p}<d$ and
\begin{align*}
Y&=(-\tau_1 xw^d+y^d)\frac{\partial}{\partial x} +(\tau_2 yw^d + x^a z^b)\frac{\partial}{\partial y}+ \tau_3 zw^d \frac{\partial}{\partial z} +\\ &(\tau_4 w^{d+1} +x^l + z^d) \frac{\partial}{\partial w}.
\end{align*}

We have that $Y \in W_0$ and $0$ is an isolated singularity of $Y$. Then Theorem \hyperref[b]{B} (a) holds.\\

\item Theorem \hyperref[b]{B} (b.2), $i=2$, $p$ divides $p+\lambda$

Besides $\lambda=q(d-1)$, $c_{23}$ and $c_{34}$ are satisfied, i.e., $s+\lambda=rd$, $\lambda=q(d-1)=sd$. Thus 
\begin{equation*}
p > q=md^2> r=m(d^2-1) > s=m(d^2-d),\lambda=m(d^3-d^2), 
\end{equation*}
for some $m \in \mathbb{Z}_{>0}$. Hence $\gcd(p,q,r,s)=1 \iff \gcd(p,m)=1$ and $p$ divides $p+\lambda$ means that $p$ divides $d^3-d^2$. So we are in the situation of Theorem \hyperref[thmB2]{B.2} (c), where $p$ divides $d^3-d^2$.\\

$\bullet$ $\tau_3,\tau_4 \neq 0$ \

In this case $\tau_3=p+q-3r$ and $\tau_4=p+q+r-3s>0$. Suppose that $\tau_3=0$; this implies that $p=m(2d^2-3)$. Since $\gcd(p,m)=1$ it follows that $m=1$. Then $p=2d^2-3, q=d^2, r=d^2-1, s=d^2-d, \lambda=d^3-d^2.$ Using polynomial division, we have $2(p+\lambda)=p(d+1)+3(d-1)$. As $p$ divides $p+\lambda$ it follows that $2d^2-3$ divides $3(d-1)$, and we obtain a contradiction since $d \geq 2$.\\

$\bullet$ Theorem \hyperref[b]{B} (a) holds true \

Set $l=\frac{p+\lambda}{p}=1+\frac{sd}{p} \in \mathbb{Z}$. We have $1 < l < d+1$. Take
\begin{align*}
Y&=x(-\tau_1w^d+ax^{l-1}+a_1 y^{d-1})\frac{\partial}{\partial x}+y(\tau_2 w^d+bx^{l-1}+b_1 y^{d-1}) \frac{\partial}{\partial y}+ \\ &z(\tau_3 w^d+ cx^{l-1}+c_1 y^{d-1}) \frac{\partial}{\partial z} +(w(\tau_4 w^d+ex^{l-1}+e_1 y^{d-1})+z^d) \frac{\partial}{\partial w}.
\end{align*}
Then $Y \in W_0$ as long as $l.a+b+c+e=0$ and $a_1 + d.b_1+c_1+e_1=0$. Furthermore, $0$ is an isolated singularity of $Y$ if and only if
\begin{align*}
\begin{vmatrix}
-\tau_1&a&a_1\\
\tau_2&b&b_1\\
\tau_3&c&c_1\\
\end{vmatrix} \neq 0,
\begin{vmatrix}
-\tau_1&a&a_1\\
\tau_2&b&b_1\\
\tau_4&e&e_1\\
\end{vmatrix} \neq 0,
\begin{vmatrix}
-\tau_1&a\\
\tau_3&c\\
\end{vmatrix} \neq 0,
\begin{vmatrix}
-\tau_1&a\\
\tau_4&e\\
\end{vmatrix} \neq 0, \\
\begin{vmatrix}
\tau_2&b_1\\
\tau_3&c_1\\
\end{vmatrix} \neq 0,
\begin{vmatrix}
\tau_2&b_1\\
\tau_4&e_1\\
\end{vmatrix} \neq 0,
\begin{vmatrix}
a&a_1\\
b&b_1\\
\end{vmatrix} \neq 0,
a \neq 0, b_1 \neq 0,
\end{align*}
where $\mid \cdot \mid$ stands for the determinant. After making the substitutions $e=-(l.a+b+c)$ and $e_1=-(a_1 + d.b_1+c_1)$, we see that the conditions above is given by a non-empty Zariski open set on $\mathbb{C}^6$ with coordinates $(a,a_1,b,b_1,c,c_1)$, which shows that $0$ is an isolated singularity of generic $Y \in W_0$.
\end{enumerate}

The case (4) is the only one where a further condition is required in order to ensure that the families contain GK foliations. In all other cases, the verification that the $\tau_j$'s are not zero is either immediate, as $\tau_2 \neq 0$ in case (4), or it can obtained with the aid of polynomial division, as $\tau_3 \neq 0$ in case (5). Moreover, since $\lambda=sd$ in all cases, we have $\tau_4=p+q+r-3s>0$. The verification that Theorem \hyperref[b]{B} (a) holds is very close to what we did in cases (4) and (5). By symmetry, in Theorem \hyperref[b]{B} (b.2), $i=3$, the families given by the case $p$ divides $q+\lambda$ coincide with those given by the case $p$ divides $s+\lambda$. We summarize in the following table the correspondence between Theorem \hyperref[b]{B} and Theorem \hyperref[thmB2]{B.2}.

\begin{center}
\resizebox{\columnwidth}{!}{%
\begin{tabular}{|c|c|c|c|c|}
\hline
Case                                                            & $p \mid p+\lambda$                                                              & $p \mid q+\lambda$                                                                  & $p \mid r+\lambda$                                                                                    & $p \mid s+\lambda$                                                                                             \\ \hline
\begin{tabular}[c]{@{}c@{}}Theorem B (b.1)\\ $i=0$\end{tabular} & \begin{tabular}[c]{@{}c@{}}Theorem B.2 (a)\\ $p$ divides $d^3$\end{tabular}     & \begin{tabular}[c]{@{}c@{}}Theorem B.2 (a)\\ $p$ divides $d^3+d^2+d+1$\end{tabular} & \begin{tabular}[c]{@{}c@{}}Theorem B (a) \\ not satisfied\\ $m_1 \not\in \mathbb{Z}$\end{tabular}     & \begin{tabular}[c]{@{}c@{}}Theorem B (a) \\ not satisfied\end{tabular}                                         \\ \hline
\begin{tabular}[c]{@{}c@{}}Theorem B (b.1)\\ $i=1$\end{tabular} & \begin{tabular}[c]{@{}c@{}}Theorem B.2 (b)\\ $k$ divides $d$\end{tabular}       & \begin{tabular}[c]{@{}c@{}}Theorem B.2 (b)\\ $kd$ divides $m(d^2+d)+k$\end{tabular} & \begin{tabular}[c]{@{}c@{}}Theorem B.2 (b)\\ $d$ divides $m$ and  \\$k$ divides $d^2+d+1$\end{tabular} & \begin{tabular}[c]{@{}c@{}}Theorem B.2 (b)\\ $k$ divides $d+1$\\ and $\gcd(\frac{m(d+1)}{k},d)=1$\end{tabular} \\ \hline
\begin{tabular}[c]{@{}c@{}}Theorem B (b.2)\\ $i=2$\end{tabular} & \begin{tabular}[c]{@{}c@{}}Theorem B.2 (c)\\ $p$ divides $d^3-d^2$\end{tabular} & \begin{tabular}[c]{@{}c@{}}Theorem B.2 (c)\\ $p$ divides $d^3$\end{tabular}         & \begin{tabular}[c]{@{}c@{}}Theorem B.2 (c)\\ $p$ divides $d^3-1$\end{tabular}                         & \begin{tabular}[c]{@{}c@{}}Theorem B (a) \\ not satisfied\\ $m_1 \not\in \mathbb{Z}$\end{tabular}              \\ \hline
\begin{tabular}[c]{@{}c@{}}Theorem B (b.2)\\ $i=3$\end{tabular} & \begin{tabular}[c]{@{}c@{}}Theorem B.2 (d)\\ $k$ divides $d-1$\end{tabular}     & \begin{tabular}[c]{@{}c@{}}Same as case $p$ \\ divides $s+\lambda$\end{tabular}     & \begin{tabular}[c]{@{}c@{}}Theorem B.2 (d)\\ $k$ divides $d$\end{tabular}                             & \begin{tabular}[c]{@{}c@{}}Theorem B.2 (d)\\ $d$ divides $m$ and\\ $k$ divides $d^2-1$\end{tabular}            \\ \hline
\end{tabular}
}
\end{center}

\end{proof}

\begin{proof}[Proof of Theorem B.1]  
The proof is very similar to that of Theorem \hyperref[thmB2]{B.2}. In Theorem \hyperref[b]{B}, $n=3$, there are three different situations: Theorem \hyperref[b]{B} (b.1), $i=0$ and $i=1$, Theorem \hyperref[b]{B} (b.2), $i=2$. These, combined with the possibilities of Remark \ref{division}, give rise to 9 cases to consider. There is only one case that we do not have GK foliations, namely Theorem \hyperref[b]{B} (b.1), $i=0$, $p$ divides $r+\lambda$. In fact, the proof that $$m_1=\frac{(p+\lambda)(q+\lambda)(r+\lambda)}{pqr} \not\in \mathbb{Z}$$ is similar to that of case (1) of the proof of Theorem \hyperref[thmB2]{B.2}.

All the other eight cases provide families containing GK foliations, and we can proceed as in the previous proof to verify that the conditions of Theorem \hyperref[b]{B} are satisfied. By symmetry, the cases corresponding to Theorem \hyperref[b]{B} (b.1), $i=1$, $p$ divides $q+\lambda$ and $p$ divides $r+\lambda$ generate the same families of foliations. We summarize in the following table the correspondence between the two theorems.

\begin{center}
\resizebox{\textwidth}{!}{%
\begin{tabular}{|c|c|c|c|}
\hline
Case                                                            & $p \mid p+\lambda$                                                            & $p \mid q+\lambda$                                                              & $p \mid r+\lambda$                                                                             \\ \hline
\begin{tabular}[c]{@{}c@{}}Theorem B (b.1)\\ $i=0$\end{tabular} & \begin{tabular}[c]{@{}c@{}}Theorem B.1 (a)\\ $p$ divides $d^2$\end{tabular}   & \begin{tabular}[c]{@{}c@{}}Theorem B.1 (a)\\ $p$ divides $d^2+d+1$\end{tabular} & \begin{tabular}[c]{@{}c@{}}Theorem B (a) not satisfied\\ $m_1 \not\in \mathbb{Z}$\end{tabular} \\ \hline
\begin{tabular}[c]{@{}c@{}}Theorem B (b.1)\\ $i=1$\end{tabular} & Theorem B.1 (b)                                                               & \begin{tabular}[c]{@{}c@{}}Same as case $p$\\ divides $r+\lambda$\end{tabular}  & Theorem B.1 (c)                                                                                \\ \hline
\begin{tabular}[c]{@{}c@{}}Theorem B (b.2)\\ $i=2$\end{tabular} & \begin{tabular}[c]{@{}c@{}}Theorem B.1 (d)\\ $p$ divides $d^2-d$\end{tabular} & \begin{tabular}[c]{@{}c@{}}Theorem B.1 (d)\\ $p$ divides $d^2$\end{tabular}     & \begin{tabular}[c]{@{}c@{}}Theorem B.1 (d)\\ $p$ divides $d^2-1$\end{tabular}                  \\ \hline
\end{tabular}%
}
\end{center}
\end{proof}

From Theorem \hyperref[thmB1]{B.1}, for instance, we display the GK components of $\mathscr{F}_2(3,3)$ and $\mathscr{F}_2(4,3)$ given by Theorem \ref{thefour}.

\begin{cor}\label{corsomecomponents}
For each $p,q,r,\lambda$, $\overline{\mathcal{F}(p,q,r;\lambda,3)}$ is an irreducible component of $\mathscr{F}_2(3,3)$
\begin{center}
\begin{tabular}{|c||*{6}{c|}}\hline 
$p$ & $7$ & $7$ & $6$ & $4$ & $4$ & $3$ \\ \hline
$q$ & $6$ & $3$ & $5$ & $3$ & $2$ & $2$ \\ \hline
$r$ & $4$ & $2$ & $2$ & $2$ & $1$ & $1$ \\ \hline
$\lambda$ & $8$ & $4$ & $4$ & $4$ & $2$ & $2$ \\ \hline
\end{tabular}
\end{center}

For each $p,q,r,\lambda$, $\overline{\mathcal{F}(p,q,r;\lambda,4)}$ is an irreducible component of $\mathscr{F}_2(4,3)$
\begin{center}
\begin{tabular}{|c||*{13}{c|}}\hline 
$p$ & $13$ & $13$ & $13$ & $12$ & $9$ & $9$ & $9$ & $9$ & $8$ & $6$ & $6$ & $4$ & $3$ \\ \hline
$q$ & $12$ & $8$ & $4$ & $7$ & $8$ & $6$ & $4$ & $3$ & $3$ & $5$ & $3$ & $3$ & $2$ \\ \hline
$r$ & $9$ & $6$ & $3$ & $3$ & $6$ & $4$ & $3$ & $2$ & $2$ & $3$ & $2$ & $2$ & $1$ \\ \hline
$\lambda$ & $27$ & $18$ & $9$ & $9$ & $18$ & $12$ & $9$ & $6$ & $6$ & $9$ & $6$ & $6$ & $3$ \\ \hline
\end{tabular}
\end{center}
\end{cor}

\begin{cor}\label{answerproblem}
If $q \geq 3$, there are no $\lambda \neq 0$ and $d \geq 1$ such that $$\overline{\mathcal{F}(q+1,q,1;\lambda,d+1)}$$ contains some GK foliation.
\end{cor}
\begin{proof}
A simple verification shows that $p=q+1,q,r=1$ do not satisfy any of the four relations of Theorem \hyperref[thmB1]{B.1}, regardless of choices for $\lambda \neq 0$ and $d \geq 2$. The result follows, since $\overline{p}=p$, $\overline{q}=q$ and $\overline{r}=r$.
\end{proof}

\begin{cor}\label{corfamilies}
For $d \geq 2$, $\overline{\mathcal{F}(p,q,r;\lambda,d+1)}$ is an irreducible component of $\mathscr{F}_2(d+1,3)$, for the following values of $p,q,r,\lambda$
\begin{center}
\begin{tabular}{|c|c|c|c|}\hline 
$p$ & $q$ & $r$ & $\lambda$ \\ \hline \hline
$d^2+d+1$ & $d+1$ & $1$ & $-1$ \\ \hline
$d^2+d+1$ & $d+1$ & $d$ & $d^2$ \\ \hline
$d^2+d$ & $2d+1$ & $d$ & $d^2$ \\ \hline
$d^2$ & $d+1$ & $d$ & $d^2$ \\ \hline
$d^2$ & $d$ & $1$ & $0$ \\ \hline
$d^2$ & $d$ & $d-1$ & $d^2-d$ \\ \hline
$d^2-1$ & $d$ & $d-1$ & $d^2-d$ \\ \hline
\end{tabular}
\end{center}
\end{cor}
\begin{proof}
In Theorem \hyperref[thmB1]{B.1}, just make the following substitutions: in (a), $m=d$, $p=d^2+d+1$ and apply Proposition \ref{symmetry}; in (a), $m=1$ and $p=d^2+d+1$; in (c), $k=d+1$ and $m=1$; in (a), $m=1$ and $p=d^2$; in (a), $m=d-1$, $p=d^2$ and apply Proposition \ref{symmetry}; in (d), $m=1$ and $p=d^2$; in (d), $m=1$ and $p=d^2-1$.
\end{proof}

Similar results can be established from Theorem \hyperref[thmB2]{B.2}. We have for instance

\begin{cor}\label{cod2comp}
For each $p,q,r,s,\lambda$, $\overline{\mathcal{F}(p,q,r,s;\lambda,3)}$ is an irreducible component of $\mathscr{F}_2(3,4)$
\begin{center}
\begin{tabular}{|c||*{10}{c|}}\hline 
$p$ & $15$ & $15$ & $14$ & $12$ & $8$ & $8$ & $7$ & $6$ & $6$ & $4$  \\ \hline
$q$ & $14$ & $7$ & $11$ & $8$ & $7$ & $4$ & $4$ & $5$ & $5$ & $3$ \\ \hline
$r$ & $12$ & $6$ & $6$ & $3$ & $6$ & $3$ & $3$ & $4$ & $3$ & $2$ \\ \hline
$s$ & $8$ & $4$ & $4$ & $2$ & $4$ & $2$ & $2$ & $2$ & $2$ & $1$ \\ \hline
$\lambda$ & $16$ & $8$ & $8$ & $4$ & $8$ & $4$ & $4$ & $4$ & $4$ & $2$ \\ \hline
\end{tabular}
\end{center}
\end{cor}

\vspace{0.3cm}

\subsection*{\emph{Acknowledgements}}
I am deeply grateful to A. Lins Neto for the several fruitful discussions. I would like also to thank R. Lizarbe, T. Fassarela and H. Movasati for the suggestions and comments on the manuscript. This work was developed at IMPA (Rio de Janeiro, Brazil) and partially supported by CNPq and CAPES.

\end{document}